\documentclass{amsart}[11pt]

\usepackage{amsmath}
\usepackage{amssymb}
\usepackage{amsthm}
\usepackage[mathscr]{eucal}
\usepackage{mathrsfs}
\usepackage{psfrag}
\usepackage{fullpage}
\usepackage{color}
\usepackage{eucal}
\usepackage[dvips]{graphicx}
\usepackage{amsfonts}%
\usepackage{amsaddr}
\usepackage{graphicx}
\usepackage{graphics}
\usepackage{enumerate}

\newcommand{\E}{\mathbb{E}}

\newcommand{\opL}{\mathcal{L}}
\newcommand{\T}{\mathsf{T}}
\newcommand{\eps}{\varepsilon}

\newtheorem{lemma}{Lemma}
\newtheorem{theorem}{Theorem}
\newtheorem{condition}{Condition}
\newtheorem{definition}{Definition}
\newtheorem{remark}{Remark}
\newtheorem{proposition}{Proposition}

\numberwithin{lemma}{section}
\numberwithin{theorem}{section}
\numberwithin{condition}{section}
\numberwithin{definition}{section}
\numberwithin{remark}{section}
\numberwithin{proposition}{section}

\title[Importance sampling for slow--fast diffusions based on moderate deviations]{Importance sampling for slow--fast diffusions based on moderate deviations}
\author{Matthew R. Morse and Konstantinos Spiliopoulos}
\address{Department of Mathematics and Statistics \\
Boston University \\
Boston, MA 02215}
\email[Matthew R. Morse]{mrmorse@bu.edu}
\email[Konstantinos Spiliopoulos]{kspiliop@math.bu.edu}
\thanks{The present research was partially supported by the National Science Foundation (DMS 1550918)}
\date{\today}
\begin{document}

\begin{abstract}
We consider systems of slow--fast diffusions with small noise in the slow component.  We construct provably logarithmic asymptotically optimal importance schemes for the estimation of rare events based on the moderate deviations principle.  Using the subsolution approach we construct schemes and identify conditions under which the schemes will be asymptotically optimal.  Moderate deviations--based importance sampling offers a viable alternative to large deviations importance sampling when the events are not too rare. In particular, in many cases of interest one can indeed construct the required change of measure in closed form, a task which is more complicated using the large deviations--based importance sampling, especially when it comes to multiscale dynamically evolving processes. The presence of multiple scales and the fact that we do not make any periodicity assumptions for the coefficients driving the processes, complicates the design and the analysis of efficient importance sampling schemes. Simulation studies illustrate the theory.
\end{abstract}

\maketitle



\section{Introduction}

The goal of this paper is to explore the properties of moderate deviations--based importance sampling for small noise multiscale diffusion processes. Importance sampling is a variance reduction technique in Monte Carlo simulation, which is in particular useful when one is interested in estimating rare events. We consider the following system of slow--fast diffusion processes driven by small noise
\begin{align}
  dX_t^\eps &= \left[ \frac{\eps}{\delta} b(X_t^\eps, Y_t^\eps) + c(X_t^\eps, Y_t^\eps) \right] \,dt + \sqrt{\eps} \sigma(X_t^\eps, Y_t^\eps) \,dW_t  \label{E:generalSDE}\\
  dY_t^\eps &= \frac{1}{\delta} \left[ \frac{\eps}{\delta} f(X_t^\eps, Y_t^\eps) + g(X_t^\eps, Y_t^\eps) \right] \,dt + \frac{\sqrt{\eps}}{\delta} \left[ \tau_1(X_t^\eps, Y_t^\eps) \,dW_t + \tau_2(X_t^\eps, Y_t^\eps) \,dB_t \right]  \notag \\
  & X_{t_0}^\eps = x_0, \quad Y_{t_0}^\eps = y_0 \notag
\end{align}
for $t \in [t_0,T]$ such that $(X_t^\eps, Y_t^\eps) \in \mathbb{R}^n \times \mathbb{R}^d$. For convenience, we refer to the state space of $Y^\eps$ as $\mathcal{Y}$. Here, $W_t$ and $B_t$ are independent $m$-dimensional Brownian motions. The parameter $\delta \ll 1$ is the time-scale separation parameter, whereas $\eps \ll 1$ controls the strength of the noise.  The scaling in  \eqref{E:generalSDE} implies that $X^\eps$ is the slow motion and $Y^\eps$ is the fast motion. In this work we assume that the initial point $(x_0,y_0)$ is given.

Depending on the order in which $\eps, \delta$ go to zero, we get different behavior. We are interested in the following two limiting regimes:
\begin{equation}\label{Eq:Regimes}
  \lim_{\eps \downarrow 0} \frac{\eps}{\delta} = \begin{cases}
     \infty, & \text{Regime 1}, \\
     \gamma \in (0, \infty), & \text{Regime 2}.
   \end{cases}
\end{equation}

Moderate deviations address the behavior of the process $X^{\eps}$ in the regime between the central limit theorem regime and the large deviation regime. In particular, letting $h(\eps) \to + \infty$ such that $\sqrt{\eps} h(\eps) \to 0$ as $\eps \downarrow 0$, we denote the law of large numbers limit by $\bar{X}_t = \lim_{\eps \downarrow 0} X_t^\eps$ (in the appropriate sense), and define the moderate deviation process to be
\begin{equation} \label{E:MDprocess}
  \eta_t^\eps = \frac{X_t^\eps - \bar{X}_t}{\sqrt{\eps} h(\eps)}.
\end{equation}

 Deriving the moderate deviations principle for the process $X_{t}^{\eps}$ amounts to deriving the large deviation principle for $\eta_t^\eps$. Notice that if $h(\eps) = 1$ then the limiting behavior of $\eta_t^\eps$ is that of the central limit theorem (CLT), studied in \cite{S14}, whereas if $h(\eps) = 1 / \sqrt{\eps}$ then we would get the large deviation result, studied in the series of papers \cite{DS12, FS, S13, S15, Veretennikov2} under various assumptions on coefficients.  Moderate deviations for systems of the form \eqref{E:generalSDE} have been studied in \cite{BF77, F78, G03} and in its full generality allowed by \eqref{E:generalSDE}-\eqref{Eq:Regimes} in  \cite{MorseSpiliopoulos2017}.

System (\ref{E:generalSDE}) represents a general system of slow-fast diffusion with small noise in the slow component. Note that it captures as special case the standard slow-fast  averaging systems if we set the coefficients $b=g=0$, see an example of this in Sections \ref{SS:TwoScaleSystem_a} and \ref{SS:TwoScaleSystem_b} and more generally \cite{BLP,PS}. When we are in Regime 1 with $b\neq 0$, then we are in the homogenization regime, where a fast component is part of the slow motion. A typical example of such a situation is the case where we choose the coefficients $f=b$, $g=c$, $\tau_1=\sigma$ and $\tau_2=0$. These choices then correspond to the overdamped Langevin equation in a rough potential, see for example Section \ref{SS:PeriodicProblem} and in particularly \cite{Z88,DSW12}.

The goal of this paper is to translate the existing theoretical results on moderate deviations to related importance sampling simulation schemes. The need to simulate rare events appears in many areas of applications, ranging from financial engineering to physics, biology, chemistry, and climate modeling. However, when one is interested in rare events, mathematical and computational challenges arise. Consider, for example a sequence $\{X^{\eps}\}_{\eps>0}$ of random elements and assume that we want to estimate $0<\mathbb{P}\left[X^{\eps}\in A\right]\ll 1$ for a set $A$ for which the event $\{X^{\eps}\in A\}$ is less likely to happen as $\eps$ gets smaller. Given that typically closed form formulas are not available and that numerical approximations are either too crude or unavailable,  one has to resort to simulation. Naive Monte Carlo simulation techniques (i.e., using the unbiased estimator $\hat{p}^{\eps}=\frac{1}{N}\sum_{j=1}^{N}1_{X^{\eps,j}\in A}$) perform rather poorly in the rare-event regime. To be precise, for naive Monte Carlo it is known that in order to achieve relative error of order one, one needs an effective sample size $N\approx 1/p^{\eps}$, e.g. see \cite{AsmussenGlynn2007}. This means that for a fixed computational
cost, relative errors grow rapidly as the event becomes more rare.  Therefore,  naive Monte Carlo is not practical for rare-event simulation and construction of accelerated Monte Carlo methods that reduce the variance of the estimators becomes important. One such method  is importance sampling.

In importance sampling one changes the measure appropriately and simulates the event of interest under the new measure.  The new simulation measure has the property that it reduces the variance of the estimator. In contrast to the vast majority of the literature, our goal in this paper is to construct changes of measure that are optimal in the moderate deviations regime rather than in the large deviations regime. Large deviations (LD)--based importance sampling for stochastic dynamical systems is a subject that has been reasonably well studied in recent years, see for example \cite{DSW12,DupuisSpiliopoulosZhou2013,DupuisWang,DupuisWang2, VandenEijndenWeare}. In particular for systems of the form \eqref{E:generalSDE}, we refer the interested reader to \cite{DSW12} for large deviations--based importance sampling in the case of Regime 1, as defined by \eqref{Eq:Regimes}. The construction of the asymptotically (as $\eps\downarrow 0$) optimal change of measure is based on subsolutions for a related partial differential equation (PDE) of Hamilton--Jacobi--Bellman (HJB) type, an  idea first introduced in  \cite{DupuisWang,DupuisWang2}. In large deviations--based importance sampling, the difficulty is in the actual constructions of appropriate subsolutions to the related HJB equation, due to the fact that such PDEs are typically highly non-linear.

The idea that we exploit in this paper is that in cases where one is interested in events that are rare, but not too rare, moderate deviations (MD)--based importance sampling (IS) may offer a useful alternative.  While LD--based importance sampling works very well when it can be implemented, MD--based IS schemes turn out to work equally well for events that are moderately rare and are often easier to implement in practice.  The reason is that the corresponding HJB equation takes a considerably simpler form and closed form subsolutions are available in many cases.

In addition, we would like to emphasize that the cost of simulation of a single trajectory is relatively high for multiscale problems. For events that are rare, efficient importance sampling schemes are a relevant strategy for reducing the total cost of simulation by reducing the relative error per sample.

To be more precise, the goal of this paper is to develop asymptotically (as $\eps\downarrow 0$) optimal changes of measure (in the sense of variance reduction) for estimation of quantities of the form
\begin{equation}
\theta(\eps)=\mathbb{E}\left[e^{-h^{2}(\eps)H(\eta^{\eps}_{T})} \right] 
\label{Eq:QuantitiesOfInterest}
\end{equation}
with $H$ a bounded and continuous function. If one is interested in accurate estimation of quantities as in   \eqref{Eq:QuantitiesOfInterest}, then one does not have any hope of closed form formulas and logarithmic large or moderate deviation estimates are too crude (since they ignore potential important prefactors) and thus simulation becomes necessary.

The choice of the scaling  $h^2(\eps)$ in (\ref{Eq:QuantitiesOfInterest})  comes from the moderate deviations scaling.  Often one is interested in estimating quantities like (\ref{Eq:QuantitiesOfInterest}) but with $1/\eps$ in place of $h^2(\eps)$, which would then correspond to the large deviations scaling.  To this end, as we shall see in  Section \ref{S:Simulations}, estimation of $\theta(\eps)$ in \eqref{Eq:QuantitiesOfInterest} can be related to MD--based importance sampling for events of the form
\begin{equation*}
\tilde{\theta}(\eps)=\mathbb{E}\left[e^{-\frac{1}{\eps}\tilde{H}(X^{\eps}_{T})} \right] 
\end{equation*}
where $\tilde{H}$ is related to the original $H$ and typically may depend on the specific true value of $\eps$ with respect to which the actual simulation is being done.

Even though large deviations--based importance sampling has been reasonably well studied, moderate deviations--based importance sampling has only received little attention. To the best of our knowledge, the only exception to this is the recent paper \cite{DupuisJohnson2017}, where the authors study moderate deviations--based importance sampling for stochastic recursive algorithms. In this paper, we address the issues that come up in such designs in the setting of multiscale diffusions as in \eqref{E:generalSDE}. The methodology that we follow in this paper in order to establish logarithmic asymptotic optimality of the suggested changes of measure is the weak convergence approach of \cite{DE97}, which turns the large deviations problem to a law of large numbers for an appropriate stochastic control problem. The main technical difficulty, which also constitutes the main theoretical contribution of this work, is to establish tightness of the appropriate controlled version of the moderate deviations process \eqref{E:MDprocess}. Notice that we do not make any periodicity assumptions on the coefficients, which means that the fast motion can take values on the whole space which makes the derivation of the needed estimates tedious at certain places.

The rest of the paper is organized as follows. In Section \ref{S:Preliminaries} we go over our notation and imposed conditions as well as the related moderate deviations theory of \cite{MorseSpiliopoulos2017} and a general discussion on importance sampling tailored to our case of interest.  In Section \ref{S:MainResult} we state and prove our main result on logarithmic asymptotic optimality of appropriate moderate deviations--based changes of measure. Section \ref{S:RelaxedConditions} discusses how one can relax some of the growth conditions on the coefficients. Section \ref{S:Simulations} contains the simulation studies that illustrate our theoretical results. Finally, the Appendix \ref{S:AppendixA} contains certain existing results that are used in our paper.

\section{Related Moderate Deviations results and preliminaries on importance sampling} \label{S:Preliminaries}

\subsection{Notation and assumptions} \label{SS:Notation}
In this section we set up notation and lay out the assumptions of the paper. We work with the canonical filtered probability space $(\Omega, \mathscr{F}, \mathbb{P})$ equipped with a filtration $\mathscr{F}_t$ that is right continuous and such that $\mathscr{F}_0$ contains all $\mathbb{P}$-negligible sets.

For given sets $A,B$, for $i,j\in\mathbb{N}$ and $\alpha>0$ we denote by $\mathcal{C}_b^{i, j + \alpha}(A \times B)$, the space of functions with $i$ bounded derivatives in $x$ and $j$ derivatives in $y$, with all partial derivatives being $\alpha$-H\"{o}lder continuous with respect to $y$, uniformly in $x$.

In regards to the coefficients of the SDE \eqref{E:generalSDE}, we assume Conditions \ref{C:growth} and \ref{C:ergodic}.

\begin{condition} \label{C:growth}
\begin{enumerate}[(i)]
\item The functions $b$ and $c$ are in $\mathcal{C}^{2,\alpha}(\mathbb{R}^n\times\mathcal{Y})$  for some $\alpha>0$ and there exist constants $0<K<\infty$ and $q_{b},q_{c}\geq 0$ such that
\begin{align*}
    \lvert b(x, y) \rvert + \lVert \nabla_x b(x, y) \rVert + \lVert \nabla_x \nabla_x b(x, y) \rVert \le K (1 + \lvert y \rvert^{q_{b}})\nonumber\\
    \lvert c(x, y) \rvert + \lVert \nabla_x c(x, y) \rVert + \lVert \nabla_x \nabla_x c(x, y) \rVert \le K (1 + \lvert y \rvert^{q_{c}})\nonumber\\
    \lvert b(x, y_1)-b(x, y_2) \rvert+\lvert c(x, y_1)-c(x, y_2) \rvert\leq K\lvert y_1-y_2 \rvert^{\alpha}\nonumber
\end{align*}
\item For every $N >0$ there exists a constant $C(N)$ such that for all $x_1, x_2 \in \mathbb{R}^n$ and $\lvert y \rvert \le N$, the diffusion matrix $\sigma$ satisfies
\begin{equation*}
    \lVert \sigma(x_1, y) - \sigma(x_2, y) \rVert \le C(N) \lvert x_1 - x_2 \rvert.
\end{equation*}
Moreover, there exists $K > 0$ and $q_{\sigma} \geq 0 $ such that
\begin{equation*}
    \lVert \sigma(x, y) \rVert \le K (1 + \lvert y \rvert^{q_{\sigma}}).
\end{equation*}
\item The functions $f(x, y)$, $g(x, y)$, $\tau_1(x, y)$, and $\tau_2(x, y)$ are $\mathcal{C}_b^{2, 2 + \alpha}(\mathbb{R}^n\times \mathcal{Y})$ for some $\alpha>0$. 
In addition, $g$ is uniformly bounded.
\end{enumerate}
\end{condition}

\begin{condition} \label{C:ergodic}
\begin{enumerate}[(i)]
\item The diffusion matrix $\tau_1 \tau_1^\T + \tau_2 \tau_2^\T$ is uniformly continuous and bounded and nondegenerate and there exist constants $\beta_{1},\beta_{2}>0$ such that
\[
0<\beta_{1}\leq\frac{\left<(\tau_{1}\tau_{1}^{\T}(x,y)+\tau_{2}\tau_{2}^{\T}(x,y))y,y\right>}{|y|^{2}}\leq \beta_{2}.
\]
\item 
Let us consider $\gamma_{0}>0$, a $d\times d$ positive matrix $\Gamma>0$ with bounded entries and a globally Lipschitz, uniformly bounded in $x$, function $\zeta(x,y)$ with Lipschitz constant $L_{\zeta}$ such that $\left<(\Gamma-L_{\zeta}I)\cdot\xi,\xi\right>\geq \gamma_{0}|\xi|^{2}$ for every $\xi\in\mathbb{R}^{d}$. Here $I$ is the $d\times d$ identity matrix. We assume that  in Regime 1
    \[
 f(x, y) =  - \Gamma y +\zeta(x,y)
\]
and in  Regime 2,
\[
 \gamma f(x, y) + g(x, y) = - \Gamma y +\zeta(x,y).
\]
The function $\zeta(x,y)$ can grow at most linearly with respect to $y$.
\end{enumerate}
\end{condition}

Next, we define the infinitesimal generators that correspond to the appropriate fast process in each one of the two regimes.  These generators arise by rescaling time by $\eps^2/\delta$ in Regime 1 or by $1/\delta$ in Regime 2 and then letting $\eps \downarrow 0$. In particular, for each Regime $i = 1, 2$, define the operator $\opL_{i,x}$ (treating $x$ as a parameter) by
\begin{align*}
  \opL_{1,x} F(y) &= ( \nabla_y F) (y) \cdot f(x, y) + \frac{1}{2} (\tau_1 \tau_1^\T + \tau_2 \tau_2^\T) (x, y) : \nabla_y \nabla_y F(y), \\ 
  \opL_{2,x} F(y) &= ( \nabla_y F) (y) \cdot (\gamma f(x, y) + g(x,y)) + \gamma \frac{1}{2} (\tau_1 \tau_1^\T + \tau_2 \tau_2^\T) (x, y) : \nabla_y \nabla_y F(y)\nonumber
\end{align*}
where the notation $A:B$ for two $k \times k$ matrices means
\begin{equation*}
  A:B = \text{Tr}[A^{\top}B]=\sum_{i=1}^k \sum_{j=1}^k a_{ij} b_{ij}.
\end{equation*}
For a $k \times k$ matrix $A$ and a $n$-dimensional vector--valued function of a $k$-dimensional vector $f(x)$ define $A : \nabla \nabla f$ as a $n$-dimensional vector where component $i$ is equal to $A : \nabla \nabla f_i$. Also, for notational convenience we sometimes collect the variables at the end of the expression and we write
\begin{equation*}
  \tau \tau^\T (x, y) = \tau(x, y) \tau(x, y)^\T.
\end{equation*}

Now, the operators $\opL_{1,x}$ and $\opL_{2,x}$ are the infinitesimal generators for the processes that play the role of the fast motion (and with respect to which averaging is being performed) in Regimes 1 and 2 respectively.  Condition \ref{C:ergodic} guarantees that the fast process in each Regime $i = 1, 2$ has a unique invariant measure, denoted by $\mu_{i, x}(dy)$, for each $x\in\mathbb{R}^{n}$.

However, because the fast motion takes values in an unbounded space, $\mathcal{Y}=\mathbb{R}^{d}$, the constants $q_{b},q_{c},q_{\sigma}$ that determine the growth of the coefficients from Condition \ref{C:growth}  need to be related in order for the subsequent tightness argument to go through. In particular, we have Condition \ref{C:Tightness}.
\begin{condition}\label{C:Tightness}
Consider the constants $q_{b},q_{c},q_{\sigma}$ from Condition \ref{C:growth}.  Then, we assume that
\begin{align*}
&\max\left\{q_{b}+ q_{\sigma},  q_{c}+q_{\sigma} \right\}< 1.
\end{align*}
\end{condition}

Sometimes, we may use the notation  $a\vee b=\max\{a,b\}$. In addition, in Regime 1, we impose the following centering condition.
\begin{condition} \label{C:center}
In  Regime 1, the drift term $b$ satisfies
\begin{equation*}
    \int_\mathcal{Y} b(x, y) \mu_{1, x}(dy) = 0.
\end{equation*}
\end{condition}

Then by the results in \cite{PV01, PV03}, which we collected in Theorem \ref{T:regularity} in the Appendix, for each $\ell \in 1, \dots, n$, there is a unique, twice differentiable function $\chi_\ell(x, y)$ in the class of functions that grows at most polynomially in $\lvert y \rvert$ that satisfies the equation
\begin{equation}  \label{E:cell}
  \opL_{1,x} \chi_\ell(x,y) = - b_\ell(x, y), \qquad \int_\mathcal{Y} \chi_\ell(x, y) \mu_{1,x}(dy) = 0, \text{ for }\ell=1,\cdots,n,
\end{equation}
where $b_{\ell}(x,y)$ is the $\ell^{\text{th}}$ component of the vector $b(x,y)=\left(b_{1}(x,y),\cdots, b_{n}(x,y)\right)$. Let us set $\chi(x,y) = (\chi_1(x,y), \dots, \chi_n(x, y))$.

Define the function $\lambda_i(x, y) \colon \mathbb{R}^n \times \mathcal{Y} \mapsto \mathbb{R}^n$ under Regime $i$ by
\begin{align*}
  \lambda_1(x, y) &= ( \nabla_y \chi ) (x, y) g(x, y) + c(x, y) \\
  \lambda_2(x, y) &= \gamma b(x, y) + c(x,y).
\end{align*}

Under Regime $i$, for any function $G(x,y)$, define the averaged function $\bar{G}$ by
\begin{equation*} 
  \bar{G}(x) = \int_{\mathcal{Y}} G(x, y) \mu_{i,x}(dy).
\end{equation*}
It follows that $\bar{G}$ inherits the continuity and differentiability properties of $G$. In particular, for each regime,
\begin{equation*}
  \bar{\lambda}_i(x) = \int_{\mathcal{Y}} \lambda_i(x, y) \mu_{i,x}(dy).
\end{equation*}

Then by an argument similar to that of Theorem 3.2 in \cite{S13}, as $\eps \downarrow 0$, in Regime $i$ we have the averaging result $X_t^\eps \to \bar{X}_t$ in probability, where $\bar{X}_t$ is defined by
\begin{equation*} \label{E:averagedSDE}
  d\bar{X}_t = \bar{\lambda}_i(\bar{X}_t) \,dt, \qquad \bar{X}_{t_0} = x_0.
\end{equation*}

\subsection{Related moderate deviations theory}\label{SS:RelatedMDResults}

In this subsection we recall the results of \cite{MorseSpiliopoulos2017} on the path space moderate deviations theory for $\{X^{\eps}_{\cdot}\}_{\eps>0}$ satisfying \eqref{E:generalSDE}.

We say that $\{ \eta^\eps \}_{\eps > 0}$, defined by \eqref{E:MDprocess}, satisfies the Laplace principle with action functional $S(\xi)$ if for any bounded continuous real-valued function $H$,
\begin{equation*}
  \lim_{\eps \downarrow 0} - \frac{1}{h^2(\eps)} \log \mathbb{E} \left[ \exp \left\{ - h^2(\eps) H(\eta^\eps) \right\} \right] = \inf_{\xi \in \mathcal{C}([t_0, T]; \mathbb{R}^n)} \left( S(\xi) + H(\xi) \right)
\end{equation*}
It is then well known, see \cite{DE97}, that if $S(\xi)$ has compact level sets (which is true in our case), then the Laplace principle is equivalent to the large deviations principle. Since $\eta^\eps$ is a rescaling of $X^\eps$, we say that $\{ X^\eps \}$ satisfies a moderate deviations principle (MDP) with action functional $S(\xi)$.

In order to state the moderate deviations principle for $\{X^\eps \}$, we need some more notation first. For Regime $i=1,2$, introduce the function $\Phi_i(x, y)$, given by the PDE
\begin{equation} \label{E:Phi}
  \opL_{i,x} \Phi_i(x, y) = - (\lambda_i(x,y) - \bar{\lambda}_i(x) ), \qquad \int_\mathcal{Y} \Phi_i(x, y) \mu_{i,x}(dy) = 0.
\end{equation}

Under our assumptions, each one of $\lambda_i-\bar{\lambda_{i}}$, for $i=1,2$, satisfy the assumptions of Theorem \ref{T:regularity}, part (iii), and thus
by Theorem \ref{T:regularity}, \eqref{E:Phi} has a unique classical solution in the class of functions which grow at most polynomially in $\lvert y \rvert$ for every $x$.

 In order to state the moderate deviations Theorem \ref{T:main}, we need to know the relative rates at which $\delta$, $\eps$, and $1 / h(\eps)$ vanish. In particular, in Regime $i$, $i = 1, 2$, define $j_i$ by
\begin{equation}\label{Eq:LimitingConstants}
  j_1 = \lim_{\eps \downarrow 0} \frac{\delta / \eps}{\sqrt{\eps} h(\eps)}, \qquad
  j_2 = \lim_{\eps \downarrow 0} \frac{\eps / \delta - \gamma}{\sqrt{\eps} h(\eps) }.
\end{equation}
$j_i$ specifies the relative rate at which $\eps/\delta$ goes to its limit and $h(\eps)$  goes to infinity. In order for a moderate deviations principle to hold in Regime $i$, we require that $j_i$ be finite.

\begin{theorem}[Theorem 2.1 of \cite{MorseSpiliopoulos2017}] \label{T:main}
Let Conditions \ref{C:growth}, \ref{C:ergodic}, \ref{C:Tightness}, and \ref{C:center} be in place. We also assume that the quantities $j_1$ and $j_2$ from (\ref{Eq:LimitingConstants}) are finite depending on whether we consider Regime 1 or 2 respectively.
Then under Regime $i$, $i = 1,2$, the process $\{X^{\eps},\eps>0\}$ from \eqref{E:generalSDE} satisfies the MDP, with the action functional $S(\xi)$ given by
\begin{equation*}
  S_{t_0 T}(\xi) = \frac{1}{2} \int_{t_0}^T \left\langle \dot{\xi}_s - \kappa \left(\bar{X}_s, \xi_{s} \right), q^{-1}(\bar{X}_s) \left(\dot{\xi}_s - \kappa \left(\bar{X}_s, \xi_{s} \right) \right) \right\rangle ds
\end{equation*}
if $\xi \in \mathcal{C}([t_0, T]; \mathbb{R}^n)$ is absolutely continuous and $\xi_{t_0}=0$, and $\infty$ otherwise.
Under Regime 1, we have
\begin{align*}
  \kappa(x, \eta) &= \left( \nabla_x \bar{\lambda}_1 \right) (x) \eta + j_1 \int_\mathcal{Y} \left( \nabla_y \Phi_1 \right) (x,y) g(x, y) \mu_{1,x}(dy) \\
  q(x) &= \int_\mathcal{Y} \left( \alpha_1 \alpha_1^\T + \alpha_2 \alpha_2^\T \right) (x, y) \mu_{1,x}(dy) \\
  \alpha_1(x,y) &= \sigma(x,y) + \left( \nabla_y \chi \right) (x, y) \tau_1(x, y), \
  \alpha_2(x,y) = \left( \nabla_y \chi \right) (x, y) \tau_2(x, y).
\end{align*}

Under Regime 2, we have
\begin{align*}
  &\kappa(x, \eta) =\left( \nabla_x \bar{\lambda}_2 \right) (x) \eta+j_2 \int_\mathcal{Y} \left[ b(x, y)  -\frac{1}{\gamma} \left( \nabla_y \Phi_2 \right) (x, y)  g(x, y)  \right] \mu_{2, x} (dy)\\
  &q(x) = \int_\mathcal{Y} \left( \alpha_1 \alpha_1^\T + \alpha_2 \alpha_2^\T \right) (x, y) \mu_{2, x}(dy) \\
    &\alpha_1(x,y) = \sigma(x,y) + \left( \nabla_y \Phi_2 \right) (x, y) \tau_1(x, y), \quad
  \alpha_2(x,y) = \left( \nabla_y \Phi_2 \right) (x, y) \tau_2(x, y),
\end{align*}
where the finite constants $j_{1}$, $j_{2}$ are defined in \eqref{Eq:LimitingConstants}.
\end{theorem}

\begin{remark}
The corresponding Conditions \ref{C:growth}, \ref{C:ergodic}, and \ref{C:Tightness} in \cite{MorseSpiliopoulos2017} are slightly different than the ones imposed here.

Condition 2.1 in \cite{MorseSpiliopoulos2017}  requires that both $b(x,y)$ and $c(x,y)$ be twice differentiable in $y$. In this paper, we assume the uniform in $x$ H\"{o}lder continuity property of $b$ and $c$ with respect to $y$ (part (i) of Condition \ref{C:growth} of this paper). By the results of Chapter 3 in \cite{Friedman}, we have that this condition is sufficient to guarantee that the solution to Poisson equations like (\ref{E:cell}) and (\ref{E:Phi}) is twice continuously differentiable in $y$. Together, with the rest of the regularity results on the solution to such Poisson equations (see Theorem \ref{T:regularity} for the general statement), we can then apply the It\^{o} formula to the solution of such Poisson equations.

%

Condition \ref{C:ergodic}  in this paper is stronger than that in \cite{MorseSpiliopoulos2017}, because here we require that $f(x,y)$ grows at most linearly in $y$. This assumption is used here in the tightness  proofs of Section \ref{SS:Tightness}. This stronger condition also implies that \eqref{E:generalSDE} has a unique strong solution, obviating the need to assume the existence of a solution, as in Condition 2.5 of \cite{MorseSpiliopoulos2017}.

Finally, the corresponding Condition \ref{C:Tightness} in \cite{MorseSpiliopoulos2017} is slightly  stronger than the one imposed here. Primarily,  this is because in \cite{MorseSpiliopoulos2017} tightness was proven making use of uniform in time bounds for the $L^{2}$ norm of the controlled moderate deviations process corresponding to \eqref{E:MDprocess}. As we  shall see in this paper, with a little  bit of extra work, one can get by with only $L^{1}$ bounds (compare (\ref{E:etaBound}) in this paper with part (i) of Lemma B.6 in \cite{MorseSpiliopoulos2017}). Hence one can weaken the requirement to the current Condition \ref{C:Tightness} at least within the setup of the present paper (consider Proposition \ref{P:tightness} in Section \ref{SS:Tightness} with $(u_{1}^{\eps},u_{2}^{\eps})=(0,0)$). In addition, one more difference in regards to this assumption is that the parameter $r$ appearing in Condition 2.3 of \cite{MorseSpiliopoulos2017} is $r=1$ here. This is consistent with the change to Condition \ref{C:ergodic}.
\end{remark}

Let us end this subsection with the observation that drove the developments of this paper. The function $\kappa(x, \eta)$ that appears in Theorem \ref{T:main} is affine in $\eta$ and the function $q(x)$ is constant in $\eta$. The dependence on $x$ does not matter here because $x=\bar{X}_{t}$ and the variable process corresponds to $\eta$. While this is expected by the nature of moderate deviations, note that  in the large deviations case, see \cite{DS12, S13}, the corresponding $\kappa$ and $q$ are nonlinear functions of the corresponding variable process $x=X^{\eps}_{t}$. The affine structure of $\kappa(x, \eta)$ is what makes the moderate deviations very appealing for the design of Monte Carlo simulation methods, as it makes the solution to the associated Hamilton-Jacobi-Bellman equation much easier to obtain and to work with.

\subsection{Generalities on importance sampling}\label{SS:GeneralitiesIS}
Let us briefly review here importance sampling. The material of this subsection is more--or--less standard, but appropriately tailored to cover our specific problem of interest.

In this paper, we are primarily interested in the estimation of quantities of the form
\begin{equation}
\theta(\eps)=\mathbb{E}\left[e^{-h^{2}(\eps)H(\eta^{\eps}_{T})} \right]
\label{Eq:QuantitiesOfInterest3}
\end{equation}

Let $\Gamma^{\eps} (t_{0},0)$ be an unbiased estimator of \eqref{Eq:QuantitiesOfInterest3} defined on some probability space with probability measure $\tilde{\mathbb{P}}$, so that
\[
\tilde{\mathbb{E}} \Gamma^{\eps} (t_{0},0)= \mathbb{E}\left[e^{-h^{2}(\eps)H(\eta^{\eps}_{T})}\right]
\]
with $\eta^{\eps}_{t_{0}}=0$.

In order to estimate $\theta(\eps)$ via Monte Carlo, one generates many independent copies of $\Gamma^{\eps} (t_{0},0)$ and the sample mean is the estimator. Due to unbiasedness, an efficient estimator is one that has the minimum second moment, which results in the minimum possible variance. Jensen's inequality guarantees that
\[
\tilde{\mathbb{E}}\left[\Gamma^{\eps} (t_{0},0)\right]^{2}\geq \left(\tilde{\mathbb{E}} \Gamma^{\eps} (t_{0},0) \right)^{2}=\theta^{2}(\eps).
\]

In addition, by Theorem \ref{T:main}, we know that
\[
\lim_{\eps\downarrow 0}-\frac{1}{h^{2}(\eps)}\log \theta(\eps)=G(t_{0},0)
\]
where
\[
G(t_{0},0)=\inf_{\xi\in \mathcal{C}([t_{0},T]; \mathbb{R}^n),\xi_{t_{0}}=0}\left\{S_{t_{0}T}(\xi)+H(\xi_{T})\right\}.
\]

Hence, Jensen's inequality together with Theorem \ref{T:main}  immediately guarantee that
\[
\lim_{\eps\downarrow 0} -\frac{1}{h^{2}(\eps)}\log \tilde{\mathbb{E}}\left[\Gamma^{\eps} (t_{0},0)\right]^{2}\leq 2 G(t_{0},0).
\]

Therefore, logarithmic asymptotical optimality for the estimator $\Gamma^{\eps} (t_{0},0)$ will follow if we prove that
\[
\lim_{\eps\downarrow 0} -\frac{1}{h^{2}(\eps)}\log \tilde{\mathbb{E}}\left[\Gamma^{\eps} (t_{0},0)\right]^{2}\geq 2 G(t_{0},0).
\]

Let us now discuss the construction of appropriate changes of measure. For notational convenience, let us set $Z=(W,B)$ to be a $2m-$dimensional Wiener process. Consider a function $u(s,\eta,y)$ to be a vector-valued function which is sufficiently smooth and introduce the family of probability measures $\tilde{\mathbb{P}}$, via the relation
\[
\frac{d\tilde{\mathbb{P}}}{d\mathbb{P}}=\exp\left\{  -\frac
{1}{2}h^{2}(\eps)\int_{t_{0}}^{T}\left\vert u(s,\eta^{\eps}_{s},Y^{\eps}_{s})\right\vert
^{2}ds+h(\eps)\int_{t_{0}}^{T}\left\langle u(s,\eta^{\eps}_{s},Y^{\eps}_{s}),dZ_{s}\right\rangle \right\}  .
\]

Then, Girsanov's theorem says that under the measure $\tilde{\mathbb{P}}$ the process from \eqref{E:generalSDE}, denoted by $(\tilde{X}^{\eps},\tilde{Y}^{\eps})$ is the unique strong solution of the SDE
\begin{align*}
  d\tilde{X}_s^\eps &= \left[ \frac{\eps}{\delta} b(\tilde{X}_s^\eps, \tilde{Y}_s^\eps) + c(\tilde{X}_s^\eps, \tilde{Y}_s^\eps) +\sqrt{\eps}h(\eps)\sigma(\tilde{X}_s^\eps, \tilde{Y}_s^\eps)u_{1}(s) \right] \,ds + \sqrt{\eps} \sigma(\tilde{X}_s^\eps, \tilde{Y}_s^\eps) \,dW_s \\
  d\tilde{Y}_s^\eps &= \frac{1}{\delta} \left[ \frac{\eps}{\delta} f(\tilde{X}_s^\eps, \tilde{Y}_s^\eps) + g(\tilde{X}_s^\eps, \tilde{Y}_s^\eps) +\sqrt{\eps}h(\eps)\tau_{1}(\tilde{X}_s^\eps, \tilde{Y}_s^\eps)u_{1}(s)+\sqrt{\eps}h(\eps)\tau_{2}(\tilde{X}_s^\eps, \tilde{Y}_s^\eps)u_{2}(s)\right] \,ds \\
  &\quad + \frac{\sqrt{\eps}}{\delta} \left[ \tau_1(\tilde{X}_s^\eps, \tilde{Y}_s^\eps) \,dW_s + \tau_2(\tilde{X}_s^\eps, \tilde{Y}_s^\eps) \,dB_s \right]  \\
  & \tilde{X}_{t_0}^\eps = x_0, \quad \tilde{Y}_{t_0}^\eps = y_0,
\end{align*}
for $s\in(t_{0},T]$ and $(u_{1}(s),u_{2}(s))$ represent the first and second component respectively of the function (with some abuse of notation)
\[
u(s,\tilde{\eta}^{\eps}_{s},\tilde{Y}^{\eps}_{s})=\left(u_{1}(s,\tilde{\eta}^{\eps}_{s},\tilde{Y}^{\eps}_{s}),u_{2}(s,\tilde{\eta}^{\eps}_{s},\tilde{Y}^{\eps}_{s})\right)=(u_{1}(s),u_{2}(s)).
\]

Therefore, under the measure $\tilde{\mathbb{P}}$ the estimator
\[
\Gamma^{\eps}(t_{0}, 0)=\exp\left\{  -h^{2}(\eps) H(\tilde{\eta}^{\eps
}_{T})\right\}  \frac{d\mathbb{P}}{d\tilde{\mathbb{P}}}(\tilde{\eta}^{\eps
}, \tilde{Y}^{\eps}),
\]
is an unbiased estimator for $\theta(\eps)$.  The simulation performance of
this estimator is characterized by the decay rate of its second moment
\begin{equation}
Q^{\eps}(t_{0},0;u)\doteq\tilde{\mathbb{E}}\left[  \exp\left\{
-2h^{2}(\eps)H(\tilde{\eta}^{\eps}_{T})\right\}  \left(  \frac{d\mathbb{P}%
}{d\tilde{\mathbb{P}}}(\tilde{\eta}^{\eps}, \tilde{Y}^{\eps})\right)  ^{2}\right].
\label{Eq:2ndMoment1}%
\end{equation}

As with all related importance sampling methods, construction of asymptotically optimal importance sampling schemes is done by appropriately choosing the function (control) $u$ in \eqref{Eq:2ndMoment1}. The goal is to be able to control
the behavior of the second moment $Q^{\eps}(t_{0},0;u)$.

As will become clear in Theorem \ref{T:UniformlyLogEfficient}, the construction of changes of measures (or equivalently of control functions $u$) that lead to asymptotically optimal changes of measures is based on
the proof of the moderate deviations theorem \ref{T:main} and on subsolutions to appropriate Hamilton--Jacobi--Bellman type PDEs. At this point, let us recall the notion of a subsolution to an HJB equation.

By general theory, \cite{FlemingSoner2006}, and the moderate deviations principle, Theorem \ref{T:main}, we obtain that  $G(s,\eta)$ is the viscosity solution to the HJB equation
\begin{align}
\partial_{s}G(s,\eta)+\Lambda(s, \eta, \nabla_{\eta}G(s,\eta))=0,\qquad G(T,\eta)=H(\eta)\label{Eq:HJBequation2}
\end{align}
where, in our case, the Hamiltonian takes the form
\[
\Lambda(s,\eta,p)=\left<\kappa(\bar{X}_{s},\eta), p\right>-\frac{1}{2}\left\vert q^{1/2}(\bar{X}_{s})p\right\vert^{2}
\]
 with $\kappa(x,\eta),q(x)$ the coefficients defined in Theorem \ref{T:main}.

Usually, optimal or nearly optimal schemes are overly complicated and difficult to implement. In these cases,  one may choose to construct sub-optimal but simpler schemes, but with guaranteed performance. Rare events associated with multiscale problems are rather complicated
and many times is it  very difficult to construct asymptotically optimal schemes. One efficient way to circumvent this difficulty is by constructing appropriate sub-optimal schemes
with precise bounds on asymptotic performance via the subsolution approach, introduced in \cite{DupuisWang2}. Let us now recall the definition of a subsolution.

\begin{definition}
\label{Def:ClassicalSubsolution} A function
$\check{U}(s,\eta):[t_{0},T]\times \mathbb{R}^{n}\mapsto\mathbb{R}$ is a
classical subsolution to the HJB equation \eqref{Eq:HJBequation2} if
\begin{enumerate}
\item $\check{U}$ is continuously differentiable,

\item $\partial_{s}\check{U}(s,\eta)+\Lambda(s, \eta, \nabla_{\eta}\check{U}(s,\eta))\geq0$ for every
$(s,\eta)\in(t_{0},T)\times\mathbb{R}^{n}$,

\item $\check{U}(T,\eta)\leq H(\eta)$ for $\eta\in\mathbb{R}^{n}$.
\end{enumerate}
\end{definition}

For illustration purposes and in order to avoid several technical problems, we will
 impose stronger regularity conditions on the subsolutions to
be considered than those of Definition
\ref{Def:ClassicalSubsolution}.

\begin{condition}
\label{Cond:ExtraReg} There exists a subsolution $\check{U}$ which has continuous derivatives up to
order $1$ in $s$ and order $2$ in $\eta$, and the first and second
derivatives in $\eta$ are uniformly bounded.
\end{condition}

\begin{remark}
We will see in Section \ref{S:RelaxedConditions} that Condition \ref{Cond:ExtraReg} can be partially relaxed. In particular, we can allow growth in the gradient of the subsolution $\check{U}$ with respect to $\eta$. However, for presentation purposes, we present the proofs of the results in the case of Condition \ref{Cond:ExtraReg} and then in Section \ref{S:RelaxedConditions} we mention the adjustments that are needed in order to weaken this condition.
\end{remark}

\begin{remark}
For comparison purposes with the large deviations case, we refer the interested reader to \cite{DSW12}. It is clear from the form that the HJB equation takes, that even construction of subsolutions becomes a rather difficult task  in the large deviations case. In the moderate deviations regime things are simpler because of the fact that $\kappa(x,\eta)$ is affine in $\eta$ and $q(x)$ does not depend on $\eta$ at all. This is important as the variable of differentiation in the HJB equation \eqref{Eq:HJBequation2} is $\eta$. In the large deviations case the corresponding $\kappa,q$ functions depend nonlinearly on the variable of differentiation.
\end{remark}

\begin{remark}
The HJB in (\ref{Eq:HJBequation2}) will typically have solution in the viscosity sense, see the classical manuscript  \cite{FlemingSoner2006} for details. Subsolutions to (\ref{Eq:HJBequation2}) in the sense of Definition \ref{Def:ClassicalSubsolution}  are clearly not unique. For example, when $H(\eta)\geq 0$ (which is typically the case), $\check{U}=0$ will always be a subsolution. As we shall see later on using $\check{U}=0$ corresponds to the naive Monte Carlo importance sampling algorithm.  However, subsolutions that lead to better importance sampling schemes typically exist and this is what we explore in this paper. See also  \cite{DupuisWang2} for a more general discussion on the topic of construction of subsolutions.
\end{remark}
\section{Statement and proof of the main result}\label{S:MainResult}

Let us now present the main result of this paper on logarithmic asymptotically optimal changes of measure for multiscale small noise diffusions.
\begin{theorem}
\label{T:UniformlyLogEfficient} Let $\{\left(X^{\eps}_{s}, Y^{\eps}_{s} \right),\eps>0\}$ be
the solution to \eqref{E:generalSDE} for $s\in[t_{0},T]$ with initial point $(x_{0}, y_{0})$ at time $t_{0}$.  Consider the moderate deviations process $\eta^{\eps}_{t}$ defined by \eqref{E:MDprocess}. Consider a non-negative, bounded and continuous
function $H:\mathbb{R}^{n}\mapsto\mathbb{R}$ and let $\check{U}(s,\eta)$ be a subsolution to the associated HJB equation \eqref{Eq:HJBequation2} according to Definition \ref{Def:ClassicalSubsolution}. Assume Conditions
\ref{C:growth}, \ref{C:ergodic}, \ref{C:Tightness}, \ref{C:center}, and \ref{Cond:ExtraReg}. Define the feedback control $u(s,\eta,y)=\left(u_{1}(s,\eta,y), u_{2}(s,\eta,y)\right)$ by
\begin{equation}
u(s,\eta,y)=\left(-\alpha_{1}^{T}(\bar{X}_{s},y)\nabla_{\eta}\check{U}(s,\eta), -\alpha_{2}^{T}(\bar{X}_{s},y)\nabla_{\eta}\check{U}(s,\eta)\right)  \label{Eq:feedback_controlReg1}
\end{equation}
with $\alpha_{1}(x,y),\alpha_{2}(x,y)$ defined in Theorem \ref{T:main}. Then, we have that
\begin{equation*}
\liminf_{\eps\rightarrow0}-\frac{1}{h^{2}(\eps)}\log Q^{\eps}(t_{0},0;u(\cdot))\geq
G(t_{0},0)+\check{U}(t_{0},0). 
\end{equation*}
\end{theorem}


In order to make clear how  subsolutions quantify performance, we make the following remark.
\begin{remark}\label{R:SubsolutionPerformance}
The subsolution property of $\check{U}$ implies  that $0\leq \check{U}(s,\eta)\leq
G(s,\eta)$ everywhere. Hence,  the scheme is logarithmic asymptotically optimal if $\check{U}(t_{0},0)=G(t_{0},0)$ at the starting point $(t_{0},0)$. Naive Monte Carlo corresponds to choosing the subsolution $\check{U}=0$. Therefore,  any subsolution scheme with
$$0\ll\check{U}(t_{0},0)\leq G(t_{0},0) $$
will outperform naive Monte Carlo measured by how close to $G$ the value of $\check{U}$  at the initial point $(t_{0},0)$ is.
\end{remark}

\begin{proof}[Proof of Theorem \ref{T:UniformlyLogEfficient}]
Let us now prove Theorem \ref{T:UniformlyLogEfficient}. In this subsection we present the main argument of the proof. For the sake of presentation necessary technical lemmas will be used here but proven later on.

In addition, we will use the notation of Theorem \ref{T:main} and omit distinguishing between the two different regimes because given the definitions in Theorem \ref{T:main} there is no difference in the proof.

Let us recall the definition
\begin{equation*}
u(s,\eta,y)=\left(-\alpha_{1}^{T}(\bar{X}_{s},y)\nabla_{\eta}\check{U}(s,\eta), -\alpha_{2}^{T}(\bar{X}_{s},y)\nabla_{\eta}\check{U}(s,\eta)\right)
\end{equation*}
with $\alpha_{1}(x,y),\alpha_{2}(x,y)$ defined in Theorem \ref{T:main}. Then, for $s\in(t_{0},T]$ and $v(\cdot)=(v_{1}(\cdot),v_{2}(\cdot))\in\mathcal{A}$, where
\[
\mathcal{A}=\left\{v(\cdot)=(v_{1}(\cdot),v_{2}(\cdot)): v \text{ is a }\mathscr{F}\text{-progressively measurable  process satisfying } \mathbb{E}\int_{t_0}^{T}\left| v(s)\right|^{2}ds<\infty\right\},
\]
define the process
 $(\hat{X}^{\eps},\hat{Y}^{\eps})$ as the unique strong solution of the SDE
\begin{align}
  d\hat{X}_s^\eps &= \left[ \frac{\eps}{\delta} b(\hat{X}_s^\eps, \hat{Y}_s^\eps) + c(\hat{X}_s^\eps, \hat{Y}_s^\eps) +\sqrt{\eps}h(\eps)\sigma(\hat{X}_s^\eps, \hat{Y}_s^\eps)\left(v_{1}(s)-u_{1}(s)\right) \right] \,ds + \sqrt{\eps} \sigma(\hat{X}_s^\eps, \hat{Y}_s^\eps) \,dW_s \notag \\
  d\hat{Y}_s^\eps &= \frac{1}{\delta} \left[ \frac{\eps}{\delta} f(\hat{X}_s^\eps, \hat{Y}_s^\eps) + g(\hat{X}_s^\eps, \hat{Y}_s^\eps) +\sqrt{\eps}h(\eps)\tau_{1}(\hat{X}_s^\eps, \hat{Y}_s^\eps)\left(v_{1}(s)-u_{1}(s)\right)+\sqrt{\eps}h(\eps)\tau_{2}(\hat{X}_s^\eps, \hat{Y}_s^\eps)\left(v_{2}(s)-u_{2}(s)\right)\right] \,ds \notag \\
  &\quad + \frac{\sqrt{\eps}}{\delta} \left[ \tau_1(\hat{X}_s^\eps, \hat{Y}_s^\eps) \,dW_s + \tau_2(\hat{X}_s^\eps, \hat{Y}_s^\eps) \,dB_s \right]  \notag \\
  & \hat{X}_{t_0}^\eps = x_0, \quad \hat{Y}_{t_0}^\eps = y_0. \label{E:generalControlledSDE}
\end{align}

Here $u(s)=u(s,\hat{\eta}_s^\eps,\hat{Y}_s^\eps)$, where $\hat{\eta}_s^\eps$ is the controlled moderate deviations process
\begin{align}
\hat{\eta}^{\eps}_{s}=\frac{\hat{X}_s^\eps-\bar{X}_{s}}{\sqrt{\eps}h(\eps)}.\label{Eq:HatModerateDeviationsProcess}
\end{align}

Then, under the imposed assumptions,  Lemma 4.3 of \cite{DSW12} guarantees the validity of the following representation
\begin{align}
-\frac{1}{h^{2}(\eps)}\log Q^{\eps}(t_{0},0;u)  =\inf_{v\in\mathcal{A}}\mathbb{E}\left[  \frac{1}{2}\int_{t_{0}}^{T}\left|
v(s)\right| ^{2}ds-\int_{t_{0}}^{T}| u(s,\hat{\eta}^{\eps}_{s},\hat{Y}^{\eps}
_{s})|^{2}ds+2H(\hat{\eta}^{\eps}_{T})\right]. \label{Eq:ToBeBounded}
\end{align}

The next step is to take $\eps\downarrow 0$. We may assume that
\begin{equation*}
  \int_0^T \lvert v(s) \rvert^2 \,ds < R, \text{ almost surely},
\end{equation*}
for some large enough constant $R<\infty$ that does not depend on $\eps$ or $\delta$,
see Lemma \ref{L:uBound}.
By Proposition \ref{P:tightness} the family $\{\hat{\eta}^{\eps}\}_{\eps>0}$ is tight on $\mathcal{C}([t_{0},T];\mathbb{R}^{n})$. Then, under the boundedness assumption of Condition \ref{Cond:ExtraReg}, the moderate deviations computations of \cite{MorseSpiliopoulos2017} go through almost verbatim (albeit a superficial difference due to the dependence of $u(s,\eta,y)$ on $s$ and $\eta$). In particular, with the definitions of Theorem \ref{T:main} for each one of the two regimes in place, let us set
\begin{align}
\tilde{\kappa}(s,x,\eta)&=\kappa(x,\eta)-\int_{\mathcal{Y}}\alpha_{1}(x,y)u_{1}(s,\eta,y)\mu_{x}(dy)-\int_{\mathcal{Y}}\alpha_{2}(x,y)u_{2}(s,\eta,y)\mu_{x}(dy)\nonumber\\
&=\kappa(x,\eta)+\int_{\mathcal{Y}}\left[\left(\alpha_{1}\alpha_{1}^{\top}+\alpha_{2}\alpha_{2}^{\top}\right)(x,y)\nabla_{\eta}\check{U}(s,\eta)\right]\mu_{x}(dy)\nonumber\\
&=\kappa(x,\eta)+q(x)\nabla_{\eta}\check{U}(s,\eta),\nonumber
\end{align}
and compute
\begin{align}
\tilde{S}_{t_{0}T}(\xi)&=\frac{1}{2} \int_{t_{0}}^{T} \left\langle \dot{\xi}_s - \tilde{\kappa} \left(s,\bar{X}_s, \xi_{s} \right) , q^{-1}(\bar{X}_s) \left(\dot{\xi}_s - \tilde{\kappa} \left(s,\bar{X}_s, \xi_{s} \right) \right) \right\rangle ds\nonumber\\
&=S_{t_{0}T}(\xi)- \int_{t_{0}}^{T} \left<\dot{\xi}_s - \kappa \left(\bar{X}_s, \xi_{s} \right), \nabla_{\eta}\check{U}(s,\xi_{s})\right>  ds+\frac{1}{2}\int_{t_{0}}^{T}\left<\nabla_{\eta}\check{U}(s,\xi_{s}), q(\bar{X}_{s})\nabla_{\eta}\check{U}(s,\xi_{s})\right>ds,\nonumber
\end{align}
and
\begin{align}
\int_{t_{0}}^{T}\int_{\mathcal{Y}}\left| u(s,\xi_{s},y)\right|^{2}\mu_{\bar{X}_{s}}(dy)ds&=
\int_{t_{0}}^{T}\left<\nabla_{\eta}\check{U}(s,\xi_{s}), q(\bar{X}_{s})\nabla_{\eta}\check{U}(s,\xi_{s})\right>ds.\nonumber
\end{align}

Thus, by following the proof of the lower bound of Theorem 2.1 in \cite{MorseSpiliopoulos2017} (Section 5.3 in \cite{MorseSpiliopoulos2017}) and making use of the previous displays we have the bound
\begin{align}
&\liminf_{\eps\rightarrow0}-\frac{1}{h^{2}(\eps)}\log Q^{\eps}(t_{0},0;u)\geq \inf_{\xi\in\mathcal{C}([t_{0},T];\mathbb{R}^{n}),\xi_{t_{0}}=0}\left[ \tilde{S}_{t_{0}T}(\xi)-\int_{t_{0}}^{T}\int_{\mathcal{Y}}\left| u(s,\xi_{s},y)\right|^{2}\mu_{\xi_{s}}(dy)ds+2H(\xi_{T})\right]\nonumber\\
&=\inf_{\xi\in\mathcal{C}([t_{0},T];\mathbb{R}^{n}),\xi_{t_{0}}=0}\left[ S_{t_{0}T}(\xi)- \int_{t_{0}}^{T} \left(\left<\dot{\xi}_s - \kappa \left(\bar{X}_s, \xi_{s} \right), \nabla_{\eta}\check{U}(s,\xi_{s})\right> +\frac{1}{2}\left<\nabla_{\eta}\check{U}(s,\xi_{s}), q(\bar{X}_{s})\nabla_{\eta}\check{U}(s,\xi_{s})\right> \right) ds\right.\nonumber\\
&\hspace{5cm}\left.
\vphantom{\int_{t_0}^T}+2H(\xi_{T})\right].\label{Eq:ToBeBoundedAfterLimit}
\end{align}

Now, using the fact that $\check{U}(s,\eta)$ satisfies the subsolution property we get
\begin{align}
\frac{d}{ds}\check{U}(s,\xi_{s})&=\partial_{s}\check{U}(s,\xi_{s})+\left<\nabla_{\eta}\check{U}(s,\xi_{s}),\dot{\xi}_{s}\right>\nonumber\\
&\geq \left<\dot{\xi}_s - \kappa \left(\bar{X}_s, \xi_{s} \right), \nabla_{\eta}\check{U}(s,\xi_{s})\right> +\frac{1}{2}\left<\nabla_{\eta}\check{U}(s,\xi_{s}), q(\bar{X}_{s})\nabla_{\eta}\check{U}(s,\xi_{s})\right>\nonumber
\end{align}
or, after integrating,
\begin{align}
\check{U}(T,\xi_{T})-\check{U}(t_{0},\xi_{t_{0}})&\geq \int_{t_{0}}^{T}\left(\left<\dot{\xi}_s - \kappa \left(\bar{X}_s, \xi_{s} \right), \nabla_{\eta}\check{U}(s,\xi_{s})\right> +\frac{1}{2}\left<\nabla_{\eta}\check{U}(s,\xi_{s}), q(\bar{X}_{s})\nabla_{\eta}\check{U}(s,\xi_{s})\right>\right)ds,\nonumber
\end{align}
and, after using again the subsolution property (the terminal condition this time), we get
\begin{align}
H(\xi_{T})-\bar{U}(t_{0},\xi_{t_{0}})&\geq \int_{t_{0}}^{T}\left(\left<\dot{\xi}_s - \kappa \left(\bar{X}_s, \xi_{s} \right), \nabla_{\eta}\bar{U}(s,\xi_{s})\right> +\frac{1}{2}\left<\nabla_{\eta}\bar{U}(s,\xi_{s}), q(\bar{X}_{s})\nabla_{\eta}\bar{U}(s,\xi_{s})\right>\right)ds.\nonumber
\end{align}

Then, finally, inserting the last display into \eqref{Eq:ToBeBoundedAfterLimit} gives the bound
\begin{align}
\liminf_{\eps\rightarrow0}-\frac{1}{h^{2}(\eps)}\log Q^{\eps}(t_{0},0;u)
&\ge \inf_{\xi\in\mathcal{C}([t_{0},T];\mathbb{R}^{n}),\xi_{t_{0}}=0}\left[ S_{t_{0}T}(\xi)+H(\xi_{T})+\check{U}(t_{0},\xi_{t_{0}})\right]\nonumber\\
&=G(t_{0},0)+\check{U}(t_{0},0),\nonumber
\end{align}
concluding the proof of the theorem.
\end{proof}

\subsection{Tightness of $\{\hat{\eta}^{\eps}\}_{\eps>0}$ on $\mathcal{C}([t_{0},T];\mathbb{R}^{n})$}\label{SS:Tightness}
In this section we prove tightness of the family $\{\hat{\eta}^{\eps}\}_{\eps>0}$. 
Sometimes, we will write $(\hat{X}^{\eps,v^{\eps}},\hat{Y}^{\eps,v^{\eps}})$ in order to emphasize the dependence on the control term.

Let us first establish the main estimates that need to be established.

\begin{lemma} \label{L:YIntegralgrowth}
Assume that Conditions~\ref{C:growth}, \ref{C:ergodic}, \ref{C:Tightness}, and \ref{Cond:ExtraReg} hold. Consider any family $\{v^{\eps},\eps>0\}$ of
controls in $\mathcal{A}$ satisfying, for some $R<\infty$,
\begin{equation*}
    \sup_{\eps > 0} \int_{t_{0}}^T | v^{\eps}(s) |^2 d s < R
\end{equation*}
almost surely. Then there exist $\eps_{0}>0$ small enough such that
 \begin{equation*}
    \sup_{\eps\in(0,\eps_{0})}\mathbb{E}  \int_{t_{0}}^{T} | \hat{Y}_s^{\eps, v^{\eps}}|^{2} d s  \le K(R,T),
\end{equation*}
and
\begin{equation*}
\frac{\delta^{2}}{\eps}\mathbb{E}\left(\sup_{t\in[t_{0},T]}\left|\hat{Y}_t^{\eps, v^{\eps}}\right|^{2}\right)\leq K(R,T)\left(1+\frac{\delta}{\sqrt{\eps}}h(\eps)\right)
\end{equation*}
for some finite constant $K(R,T)$ that may depend on $(R,T)$, but not on $\eps,\delta$. 
\end{lemma}
\begin{proof}[Proof of Lemma \ref{L:YIntegralgrowth}]
To simplify notation, we set (without loss of generality) $g(x,y)=\tau_{2}(x,y)=0$,  $t_{0}=0$ and rename $\tau_{1}=\tau$. By Condition \ref{C:ergodic}, we can write that
\[
f(x,y)=-\Gamma y+\zeta(x,y)
\]
such that $\zeta(x,y)$ is globally Lipschitz in $y$, uniformly bounded in $x$, with Lipschitz constant $L_\zeta$ such that $\Gamma-L_{\zeta} I>0$. Then we can write
\begin{align}
\hat{Y}_t^\eps &= y_{0}+ \int_{0}^{t}\left[-\frac{\eps}{\delta^{2}} \Gamma \hat{Y}_s^\eps + \frac{\eps}{\delta^{2}} \zeta(\hat{X}_s^\eps, \hat{Y}_s^\eps) +\frac{\sqrt{\eps}h(\eps)}{\delta}\tau(\hat{X}_s^\eps, \hat{Y}_s^\eps)\left(v_{1}(s)-u_{1}(s)\right)\right] \,ds \nonumber\\
  &\quad + \frac{\sqrt{\eps}}{\delta} \int_{0}^{t} \tau(\hat{X}_s^\eps, \hat{Y}_s^\eps) \,dW_s.\notag
\end{align}

We can rewrite this as follows
\begin{align}
\hat{Y}_t^\eps &= e^{-\frac{\eps}{\delta^{2}}\Gamma t}y_{0}+ \frac{\eps}{\delta^{2}}\int_{0}^{t}e^{-\frac{\eps}{\delta^{2}}\Gamma (t-s)} \zeta(\hat{X}_s^\eps, \hat{Y}_s^\eps) ds +\frac{\sqrt{\eps}h(\eps)}{\delta}\int_{0}^{t}e^{-\frac{\eps}{\delta^{2}}\Gamma (t-s)} \tau(\hat{X}_s^\eps, \hat{Y}_s^\eps)\left(v_{1}(s)-u_{1}(s)\right) \,ds \nonumber\\
  &\quad + \frac{\sqrt{\eps}}{\delta} \int_{0}^{t}e^{-\frac{\eps}{\delta^{2}}\Gamma (t-s)} \tau(\hat{X}_s^\eps, \hat{Y}_s^\eps) \,dW_s.\notag
\end{align}

Now let us define
\begin{align}
Z^{\eps}_{t}&=\frac{\sqrt{\eps}h(\eps)}{\delta}\int_{0}^{t}e^{-\frac{\eps}{\delta^{2}}\Gamma (t-s)} \tau(\hat{X}_s^\eps, \hat{Y}_s^\eps)\left(v_{1}(s)-u_{1}(s)\right) \,ds\nonumber\\
M^{\eps}_{t}&=\frac{\sqrt{\eps}}{\delta} \int_{0}^{t}e^{-\frac{\eps}{\delta^{2}}\Gamma (t-s)} \tau(\hat{X}_s^\eps, \hat{Y}_s^\eps) \,dW_s\nonumber\\
\Delta^{\eps}_{t}&=\hat{Y}_t^\eps-Z^{\eps}_{t}-M^{\eps}_{t}.\nonumber
\end{align}

A simple computation shows that
\[
d\Delta^{\eps}_{t}=-\frac{\eps}{\delta^{2}}\Gamma \cdot\Delta^{\eps}_{t}\,dt+\frac{\eps}{\delta^{2}}  \zeta(\hat{X}_s^\eps, \hat{Y}_s^\eps)\,dt .
\]

Consequently, we obtain that
\begin{align}
\frac{1}{2} d |\Delta^{\eps}_{t}|^{2}&= \left< d \Delta^{\eps}_{t}, \Delta^{\eps}_{t}\right>\nonumber\\
&\leq -\frac{\eps}{\delta^{2}}\left<\Gamma \cdot \Delta^{\eps}_{t},\Delta^{\eps}_{t}\right>,dt+\frac{\eps}{\delta^{2}} \left< \zeta(\hat{X}_t^\eps, \Delta^{\eps}_{t}+Z^{\eps}_{t}+M^{\eps}_{t})-\zeta(\hat{X}_t^\eps, Z^{\eps}_{t}+M^{\eps}_{t}), \Delta^{\eps}_{t}\right>dt\nonumber\\
&\quad +\frac{\eps}{\delta^{2}} \left< \zeta(\hat{X}_t^\eps, Z^{\eps}_{t}+M^{\eps}_{t}), \Delta^{\eps}_{t}\right>dt\nonumber\\
&\leq -\frac{\eps}{\delta^{2}}\left<(\Gamma-L_{\zeta} I) \cdot\Delta^{\eps}_{t}, \Delta^{\eps}_{t}\right>dt  +\frac{\eps}{\delta^{2}} \left< \zeta(\hat{X}_t^\eps, Z^{\eps}_{t}+M^{\eps}_{t}), \Delta^{\eps}_{t}\right>dt\nonumber\\
&\leq -\frac{\eps}{\delta^{2}}\gamma_{0} |\Delta^{\eps}_{t}|^{2}\,dt  +\frac{1}{2}\frac{\eps}{\delta^{2}} \gamma_{0}|\Delta^{\eps}_{t}|^{2}dt+C_{0}\frac{\eps}{\delta^{2}}\left|\zeta(\hat{X}_t^\eps, Z^{\eps}_{t}+M^{\eps}_{t})\right|^{2}dt\nonumber\\
&\leq -\frac{\eps}{\delta^{2}}\gamma_{0} |\Delta^{\eps}_{t}|^{2}\,dt  +\frac{1}{2}\frac{\eps}{\delta^{2}} \gamma_{0}|\Delta^{\eps}_{t}|^{2}dt+C_{0}\frac{\eps}{\delta^{2}}\left(1+ |Z^{\eps}_{t}|^{2}+|M^{\eps}_{t}|^{2}\right)dt\nonumber\\
&\leq -\frac{1}{2}\frac{\eps}{\delta^{2}}\gamma_{0} |\Delta^{\eps}_{t}|^{2}\,dt +C_{0}\frac{\eps}{\delta^{2}}\left(1+ |Z^{\eps}_{t}|^{2}+|M^{\eps}_{t}|^{2}\right)dt\nonumber
\end{align}
for some unimportant constant $C_{0}<\infty$ that may change from line to line. In the second to the last step we used Young's inequality. Therefore, by integration, we obtain
\begin{align}
|\Delta^{\eps}_{t}|^{2}&\leq e^{-\frac{\eps}{\delta^{2}}\gamma_{0}t}|y_{0}|^{2}+C_{0}\frac{\eps}{\delta^{2}}\int_{0}^{t}e^{-\frac{\eps}{\delta^{2}}\gamma_{0}(t-s)}
\left(1+ |Z^{\eps}_{s}|^{2}+|M^{\eps}_{s}|^{2}\right)ds\nonumber
\end{align}

Young's convolution inequality then yields,
\begin{align}
\int_{0}^{T}|\Delta^{\eps}_{t}|^{2} dt&\leq C_{0}\frac{\delta^{2}}{\eps}|y_{0}|^{2}+C_{0}\int_{0}^{T}\left(1+ |Z^{\eps}_{s}|^{2}+|M^{\eps}_{s}|^{2}\right)ds.\nonumber
\end{align}

Given the definition of $M^{\eps}_{t}$ we also have
\begin{align}
\mathbb{E}|M^{\eps}_{t}|^{2}&\leq \frac{\eps}{\delta^{2}} \int_{0}^{t}e^{-2\frac{\eps}{\delta^{2}}\gamma_{0} (t-s)} \mathbb{E}\left\Vert\tau(\hat{X}_s^\eps, \hat{Y}_s^\eps)\right\Vert^{2} ds,\nonumber
\end{align}
and again Young's inequality for convolutions gives
\begin{equation*}
\int_{0}^{T}\mathbb{E}|M^{\eps}_{t}|^{2} dt \leq C_{0} \int_{0}^{T} \mathbb{E}\left\Vert\tau(\hat{X}_s^\eps, \hat{Y}_s^\eps)\right\Vert^{2} ds .
\end{equation*}

Similarly, we also have for $Z^{\eps}_{t}$
\begin{align}
\int_{0}^{T}|Z^{\eps}_{t}|^{2}dt&\leq \frac{\eps h^{2}(\eps)}{\delta^{2}}\left(\int_{0}^{T}e^{-\frac{\eps}{\delta^{2}}\gamma_{0} s} ds \right)^{2}\left|\int_{0}^{T}\left|\tau(\hat{X}_s^\eps, \hat{Y}_s^\eps)\left(v_{1}(s)-u_{1}(s)\right)\right|ds\right|^{2}\nonumber\\
&\leq \frac{\delta^{2}}{\eps}h^{2}(\eps) \sup_{s\leq T}\left\Vert\tau(\hat{X}_s^\eps, \hat{Y}_s^\eps)\right\Vert^{2}\int_{0}^{T}|v_{1}(s)-u_{1}(s)|^{2}ds\nonumber\\
&\leq N \frac{\delta^{2}}{\eps}h^{2}(\eps) \sup_{s\leq T}\left\Vert\tau(\hat{X}_s^\eps, \hat{Y}_s^\eps)\right\Vert^{2}+  \frac{\delta^{2}}{\eps}h^{2}(\eps) \sup_{s\leq T}\left\Vert\tau(\hat{X}_s^\eps, \hat{Y}_s^\eps)\right\Vert^{2}\int_{0}^{T}\left|\hat{Y}_s^\eps\right|^{2}ds.\nonumber
\end{align}

In the last step we used the fact that $u_{1}(s)=-\alpha^{\top}_{1}(\bar{X}_{s},y)\nabla_{\eta}\check{U}(s,\eta)$ is bounded (by assumption) with respect to $\eta$ and $s$ and grows at most linearly with respect to $y$ (Conditions \ref{Cond:ExtraReg} and \ref{C:Tightness} respectively). Combining the latter estimates, we obtain for some unimportant constant $C_{0}<\infty$ that may change from line to line
\begin{align*}
\mathbb{E}\int_{0}^{T}|\Delta^{\eps}_{t}|^{2}dt&\leq  C_{0}\left[1+ \frac{\delta^{2}}{\eps}|y_{0}|^{2} +N \frac{\delta^{2}}{\eps}h^{2}(\eps)\mathbb{E} \sup_{s\leq T}\left\Vert\tau(\hat{X}_s^\eps, \hat{Y}_s^\eps)\right\Vert^{2}+  \frac{\delta^{2}}{\eps}h^{2}(\eps)\mathbb{E} \sup_{s\leq T}\left\Vert\tau(\hat{X}_s^\eps, \hat{Y}_s^\eps)\right\Vert^{2}\int_{0}^{T}\left|\hat{Y}_s^\eps\right|^{2}ds\right.\nonumber\\
&\qquad \qquad \left.+
\int_{0}^{T} \mathbb{E}\left\Vert\tau(\hat{X}_s^\eps, \hat{Y}_s^\eps)\right\Vert^{2} ds\right]
\end{align*}

Recalling now that $\hat{Y}_t^\eps=\Delta^{\eps}_{t}+Z^{\eps}_{t}+M^{\eps}_{t}$ we also get that for some unimportant constant $C_{0}<\infty$
\begin{align}
\mathbb{E}\int_{0}^{T}| \hat{Y}_s^\eps|^{2}\,ds&\leq  C_{0}\left[1+ \frac{\delta^{2}}{\eps}|y_{0}|^{2} +N \frac{\delta^{2}}{\eps}h^{2}(\eps) \mathbb{E}\sup_{s\leq T}\left\Vert\tau(\hat{X}_s^\eps, \hat{Y}_s^\eps)\right\Vert^{2}+  \frac{\delta^{2}}{\eps}h^{2}(\eps) \mathbb{E}\sup_{s\leq T}\left\Vert\tau(\hat{X}_s^\eps, \hat{Y}_s^\eps)\right\Vert^{2}\int_{0}^{T}\left|\hat{Y}_s^\eps\right|^{2}ds\right.\nonumber\\
&\qquad \qquad \left.+
\int_{0}^{T} \mathbb{E}\left\Vert\tau(\hat{X}_s^\eps, \hat{Y}_s^\eps)\right\Vert^{2} ds\right]\label{Eq:BoundL2normY}
\end{align}

Next using the uniform boundedness assumption on the diffusion coefficient $\tau$ and choosing $\eps$ sufficiently small so that $C_{0}\frac{\delta^{2}}{\eps}h^{2}(\eps) \sup_{x,y}\left\Vert\tau(x,y)\right\Vert^{2}<1/2$ we conclude the proof of the first statement of the  lemma.

In regards to the second statement of the lemma, we have the following calculations. By the It\^{o} formula we have
\begin{align}
|\hat{Y}_t^\eps|^{2} &= |y_{0}|^{2}+ \int_{0}^{t}2\left\langle -\frac{\eps}{\delta^{2}} \Gamma \hat{Y}_s^\eps + \frac{\eps}{\delta^{2}} \zeta(\hat{X}_s^\eps, \hat{Y}_s^\eps) +\frac{\sqrt{\eps}h(\eps)}{\delta}\tau(\hat{X}_s^\eps, \hat{Y}_s^\eps)\left(v_{1}(s)-u_{1}(s)\right) , \hat{Y}_s^\eps \right\rangle \,ds \nonumber\\
  &\quad +\frac{\eps}{\delta^{2}} \int_{0}^{t} [ \tau : \tau ] (\hat{X}_s^\eps, \hat{Y}_s^\eps) ds+ 2\frac{\sqrt{\eps}}{\delta} \int_{0}^{t} \left<\hat{Y}_s^\eps, \tau(\hat{X}_s^\eps, \hat{Y}_s^\eps) \,dW_s\right>.\notag
\end{align}

Using the Burkholder--Davis--Gundy inequality, the latter display implies, for Regime 1 (and analogously for Regime 2), that for some constant $C_{0}<\infty$ that may change from line to line
\begin{align}
\mathbb{E}\left(\sup_{t\in[0,T]}|\hat{Y}_t^\eps|^{2}\right) &\leq |y_{0}|^{2}+ \frac{\eps}{\delta^{2}}C_{0}\mathbb{E}\int_{0}^{T}\left(1+|\hat{Y}_s^\eps|^{2}\right)ds
+\frac{\sqrt{\eps}h(\eps)}{\delta}\mathbb{E}\int_{0}^{T}
 \left| \left\langle \tau(\hat{X}_s^\eps, \hat{Y}_s^\eps)\left(v_{1}(s)-u_{1}(s)\right) , \hat{Y}_s^\eps \right\rangle \right| \,ds \nonumber\\
  &\quad + 2\frac{\sqrt{\eps}}{\delta} C_{0}\mathbb{E}\left(\int_{0}^{T} \left|\hat{Y}_s^\eps\right|^{2} ds\right)^{1/2}\notag\\
   &\leq |y_{0}|^{2}+ \frac{\eps}{\delta^{2}}C_{0}\mathbb{E}\int_{0}^{T}\left(1+|\hat{Y}_s^\eps|^{2}\right)ds
+\frac{\sqrt{\eps}h(\eps)}{\delta}\left(\mathbb{E}\int_{0}^{T}
 \left|\hat{Y}_s^\eps\right|^{2}ds \right)^{1/2}\left(\mathbb{E}\int_{0}^{T} \left|v_{1}(s)-u_{1}(s)\right|^{2} \,ds\right)^{1/2} \nonumber\\
  &\quad + 2\frac{\sqrt{\eps}}{\delta} C_{0}\mathbb{E}\left(\int_{0}^{T} \left|\hat{Y}_s^\eps\right|^{2} ds\right)^{1/2}\notag\\
  &\leq |y_{0}|^{2}+ C_{0}\left[1+\frac{\eps}{\delta^{2}}+\frac{\sqrt{\eps}}{\delta}
+\frac{\sqrt{\eps}h(\eps)}{\delta}\left(\mathbb{E}\int_{0}^{T} \left|\hat{Y}_s^\eps\right|^{2(q_{\sigma}\vee q_{b})} \,ds \right)^{1/2}\right]\nonumber\\
&\leq |y_{0}|^{2}+ C_{0}\left[1+\frac{\eps}{\delta^{2}}
+\frac{\sqrt{\eps}h(\eps)}{\delta} \right].\nonumber
 \end{align}

In order to obtain the last bounds, we used the first statement of the lemma  for the integral moments of the $\hat{Y}_t^\eps$ process together with the uniform boundedness of $\tau$ and the assumption $(q_{\sigma}\vee q_{b})<1$. Then, the desired bound follows, completing the proof of the lemma.
\end{proof}

\begin{proposition}
\label{P:tightness}Assume that Conditions \ref{C:growth}, \ref{C:ergodic}, \ref{C:Tightness}, \ref{C:center}, and \ref{Cond:ExtraReg} are satisfied.  Consider any family $\{v^{\eps},\eps>0\}$ of
controls in $\mathcal{A}$ satisfying, for some $R<\infty$,
\begin{equation*}
\sup_{\eps>0}\int_{t_0}^{T}\left| v^{\eps}(t)\right|
^{2}d t<R, \text{almost surely}. 
\end{equation*}

Let $\hat{\eta}^{\eps}$ be the moderate deviations process defined  in \eqref{Eq:HatModerateDeviationsProcess} which is associated to the process $(\hat{X}^{\eps},\hat{Y}^{\eps})$ driven by the control process $v=v^{\eps}$.
Then the  family $\{\hat{\eta}^{\eps},\eps>0\}$ is tight on $\mathcal{C}([t_0,T];\mathbb{R}^{n})$.
\end{proposition}

\begin{proof}[Proof of Proposition~\ref{P:tightness}]

We will write $\hat{\eta}^{\eps,v^{\eps}}$ instead of $\hat{\eta}^{\eps}$ in order to emphasize the dependence on the control process $v^{\eps}$. In order to prove tightness of $\{ \hat{\eta}^{\eps, v^\eps},\eps>0 \}$ on $\mathcal{C}([t_0,T];\mathbb{R}^{n})$, we make use of  the characterization of Theorem 8.7 in~\cite{Billingsley1968}. It follows from that result that it is enough to prove that  there is $\eps_{0}>0$ such that for every $k>0$,
\begin{enumerate}[(i)]
\item there exists $N<\infty$ such that
\begin{equation}
\mathbb{P}\left(\sup_{t_{0}\leq t\leq T}\left|\hat{\eta}^{\eps,v^{\eps}}_{t}\right|>N\right) \leq k,
\quad\textrm{ for every }\eps\in(0,\eps_{0});
\label{Eq:CompactContainment}
\end{equation}
\item for every $M<\infty$,
\begin{equation}
\lim_{\rho\downarrow0}\sup_{\eps\in(0,\eps_{0})}\mathbb{P}\left(
\sup_{|t_{1}-t_{2}|<\rho,t_{0}\leq t_{1}<t_{2}\leq T}|\hat{\eta}^{\eps,v^{\eps}}_{t_{1}}-\hat{\eta}^{\eps,v^{\eps}}_{t_{2}}|\geq k, \sup_{t\in[t_0,T]}\left|\hat{\eta}^{\eps,v^{\eps}}_{t}\right|\leq M\right)  =0.\label{Eq:ContinuousReg}
\end{equation}
\end{enumerate}

Let us first write out what $\hat{\eta}^{\eps,v^{\eps}}$ is. For regime $i=1,2$ we have
\begin{align}
  \hat{\eta}_{t}^{\eps, v^{\eps}} - \hat{\eta}_{t_{0}}^{\eps, v^{\eps}} &= \int_{t_0}^{t} \frac{\frac{\eps}{\delta}b(\hat{X}_s^{\eps, v^{\eps}}, \hat{Y}_s^{\eps, v^{\eps}}) + c(\hat{X}_s^{\eps, v^{\eps}}, \hat{Y}_s^{\eps, v^{\eps}}) - \lambda_i(\hat{X}_s^{\eps, v^{\eps}}, \hat{Y}_s^{\eps, v^{\eps}})}{\sqrt{\eps} h(\eps)} ds \notag \\
  &+  \int_{t_0}^{t} \frac{\lambda_i(\hat{X}_s^{\eps, v^{\eps}}, \hat{Y}_s^{\eps, v^{\eps}}) - \bar{\lambda}_i(\hat{X}_s^{\eps, v^{\eps}})}{\sqrt{\eps} h(\eps)} ds + \int_{t_0}^{t} \frac{\bar{\lambda}_i(\hat{X}_s^{\eps, v^{\eps}}) - \bar{\lambda}_i(\bar{X}_s)}{\sqrt{\eps} h(\eps)} ds\notag\\
  &+ \int_{t_0}^{t} \sigma(\hat{X}_s^{\eps, v^{\eps}}, \hat{Y}_s^{\eps, v^{\eps}}) v_1^\eps(s) \,ds - \int_{t_0}^{t} \sigma(\hat{X}_s^{\eps, v^{\eps}}, \hat{Y}_s^{\eps, v^{\eps}}) u_1^\eps(s) \,ds\notag
    + \frac{1}{h(\eps)} \int_{t_0}^{t} \sigma(\hat{X}_{s}^{\eps, v^{\eps}}, \hat{Y}_{s}^{\eps, v^{\eps}}) \,dW_{s}. \notag\\
    &=\sum_{i=1}^{6} \hat{\eta}_{t}^{i,\eps} \label{Eq:ControlledProcesses}
\end{align}

where $\hat{\eta}_{t}^{i,\eps}$ represents the $i^{\text{th}}$ term on the right-hand side of~\eqref{Eq:ControlledProcesses}. We can write
\[
\mathbb{P}\left(\sup_{t_{0}\leq t\leq T}\left|\hat{\eta}^{\eps,v^{\eps}}_{t}\right|>N\right)\leq \sum_{i=1}^{6}\mathbb{P}\left(\sup_{t_{0}\leq t\leq T}\left|\hat{\eta}^{i,\eps}_{t}\right|>N/6\right)
\]

Now both statements follow from the control representation~\eqref{Eq:ControlledProcesses} together with the results on the growth of the solution to the Poisson equation by Theorem~\ref{T:regularity} and using Lemma~\ref{L:productBound} to treat each term on the right-hand side of~\eqref{Eq:ControlledProcesses}. For sake of completeness, we prove the first statement~\eqref{Eq:CompactContainment}.
The second statement~\eqref{Eq:ContinuousReg} follows similarly using the general purpose Lemma~\ref{L:productBound}. At this point we remark that the proof of Lemma \ref{L:productBound} is based on  the first statement of Lemma \ref{L:YIntegralgrowth}.

We focus on Regime 1, as Regime 2 is similar and a little bit simpler. We also set $t_{0}=0$. Let us first treat the first term, i.e, the term
\[
\hat{\eta}^{1,\eps}_{t}=\int_{0}^{t} \frac{\frac{\eps}{\delta}b(\hat{X}_s^{\eps, v^{\eps}}, \hat{Y}_s^{\eps, v^{\eps}}) + c(\hat{X}_s^{\eps, v^{\eps}}, \hat{Y}_s^{\eps, v^{\eps}}) - \lambda_1(\hat{X}_s^{\eps, v^{\eps}}, \hat{Y}_s^{\eps, v^{\eps}})}{\sqrt{\eps} h(\eps)} ds
\]

We apply the It\^{o} formula to $\chi(x,y)$, the solution to \eqref{E:cell}, with $(x,y)=(\hat{X}_s^{\eps, v^{\eps}}, \hat{Y}_s^{\eps, v^{\eps}})$ and rearrange terms to obtain
\begin{align} \label{E:coefsExpand}
  \hat{\eta}^{1,\eps}_{t}&=\frac{1}{\sqrt{\eps}h(\eps)} \int_0^t \left[ \frac{\eps}{\delta} b(\hat{X}_s^{\eps,v^\eps}, \hat{Y}_s^{\eps,v^\eps}) + c(\hat{X}_s^{\eps,v^\eps}, \hat{Y}_s^{\eps,v^\eps}) - \lambda_1(\hat{X}_s^{\eps,v^\eps}, \hat{Y}_s^{\eps,v^\eps}) \right] \,ds \\
  \notag &= - \frac{\delta}{\sqrt{\eps}h(\eps)} \left( \chi(\hat{X}_t^{\eps,v^\eps}, \hat{Y}_t^{\eps,v^\eps}) - \chi(x_0, y_0) \right) \\
  \notag &+ \frac{\delta}{\sqrt{\eps}h(\eps)} \int_0^t \left( \nabla_x \chi \right) (\hat{X}_s^{\eps,v^\eps}, \hat{Y}_s^{\eps,v^\eps}) \left[ \frac{\eps}{\delta} b(\hat{X}_s^{\eps,v^\eps}, \hat{Y}_s^{\eps,v^\eps}) + c(\hat{X}_s^{\eps,v^\eps}, \hat{Y}_s^{\eps,v^\eps}) \right] \,ds \\
  \notag &+ \delta \int_0^t \left( \nabla_x \chi \right) (\hat{X}_s^{\eps,v^\eps}, \hat{Y}_s^{\eps,v^\eps}) \sigma(\hat{X}_s^{\eps,v^\eps}, \hat{Y}_s^{\eps,v^\eps}) (v_1^\eps(s)-u_1^\eps(s)) \,ds \\
  \notag &+ \int_0^t \left( \nabla_y \chi \right) (\hat{X}_s^{\eps,v^\eps}, \hat{Y}_s^{\eps,v^\eps}) \left[ \tau_1(\hat{X}_s^{\eps,v^\eps}, \hat{Y}_s^{\eps,v^\eps}) (v_1^\eps(s)-u_1^\eps(s)) + \tau_2(\hat{X}_s^{\eps,v^\eps}, \hat{Y}_s^{\eps,v^\eps}) (v_2^\eps(s)-u_2^\eps(s)) \right] \,ds \\
  \notag &+ \frac{\delta \sqrt{\eps}}{2h(\eps)} \int_0^t \sigma \sigma^\T (\hat{X}_s^{\eps,v^\eps}, \hat{Y}_s^{\eps,v^\eps}) : \nabla_x \nabla_x \chi(\hat{X}_s^{\eps,v^\eps}, \hat{Y}_s^{\eps,v^\eps}) \,ds \\
  \notag &+ \frac{\delta}{h(\eps)} \int_0^t \left( \left( \nabla_x \chi \right) (\hat{X}_s^{\eps,v^\eps}, \hat{Y}_s^{\eps,v^\eps}) \sigma(\hat{X}_s^{\eps,v^\eps}, \hat{Y}_s^{\eps,v^\eps}) + \frac{1}{\delta} \left( \nabla_y \chi \right) (\hat{X}_s^{\eps,v^\eps}, \hat{Y}_s^{\eps,v^\eps}) \tau_1(\hat{X}_s^{\eps,v^\eps}, \hat{Y}_s^{\eps,v^\eps}) \right) \,dW_s \\
  \notag &+ \frac{1}{h(\eps)} \int_0^t \left( \nabla_y \chi \right) (\hat{X}_s^{\eps,v^\eps}, \hat{Y}_s^{\eps,v^\eps}) \tau_2(\hat{X}_s^{\eps,v^\eps}, \hat{Y}_s^{\eps,v^\eps}) \,dB_s\nonumber\\
  &=\sum_{j=1}^{7}\hat{\eta}^{1,j,\eps}_{t}, \nonumber
\end{align}
where $\hat{\eta}^{1,j,\eps}_{t}$ is the $j^{th}$ term on the right hand side of the last display.

By the second statement of Lemma \ref{L:YIntegralgrowth}, we have after an application of H\"{o}lder's inequality
\begin{align}
\mathbb{E}\left(\sup_{t\in[0,T]}| \hat{\eta}^{1,1,\eps}_{t}|\right)
 & \leq \frac{2\delta}{\sqrt{\eps} h(\eps)}\left(1+\mathbb{E}\sup_{t\in[0,T]}| Y^{\eps,v^{\eps}}_{t}|^{q_{b}}\right)
 \leq \frac{2\delta}{\sqrt{\eps} h(\eps)}\left(1+\left(\mathbb{E}\sup_{t\in[0,T]}| Y^{\eps,v^{\eps}}_{t}|^{2}\right)^{q_{b}/2}\right)\nonumber\\
 &\leq C \frac{\delta}{\sqrt{\eps} h(\eps)}\left(1+\left(\frac{\eps}{\delta^{2}}+\frac{\sqrt{\eps}}{\delta}h(\eps)\right)^{q_{b}/2}\right)
 \leq C \left(\frac{\delta}{\sqrt{\eps} h(\eps)}+\left(\frac{1}{h(\eps)^{2}}+\frac{\delta}{\sqrt{\eps}h(\eps)}\right)^{1/2}\right),\label{Eq:Eta_1_1term}
 \end{align}
which goes to zero,  hence it is certainly bounded. Next, let us treat terms $\hat{\eta}^{1,j,\eps}_{t}$ for $j=2,\dots, 7$.

Let us first look at terms $\hat{\eta}^{1,j,\eps}_{t}$ for $j=2,\dots, 5$. These are Riemann integral terms and ignoring the prefactors involving $\eps$ and $\delta$ (notice that all the prefactors go to zero apart from the term $j=4$ which has a prefactor of one) are of the form
\begin{align*}
\int_{0}^{t}B_{1}(\hat{X}_s^{\eps,v^\eps}, \hat{Y}_s^{\eps,v^\eps})ds \text{ or } \int_{0}^{t}A_{1}(\hat{X}_s^{\eps,v^\eps}, \hat{Y}_s^{\eps,v^\eps})v_{i}^{\eps}(s)ds
\end{align*}
for appropriate vector valued functions $B_{1}(x,y)$ and matrix valued functions $A_{1}(x,y)$ and $i=1,2$.
Now due to the growth assumption of Assumption \ref{C:growth} and Theorem \ref{T:regularity} we notice that
\[
|B_{1}(x,y)|\leq K(1+|y|^{q_{B_{1}}}), \text{ with }q_{B_{1}}=\max\{2q_{b},2q_{\sigma}+q_{b},q_{b}+q_{c},q_{\sigma}+q_{b}+q_{\sigma}\vee q_{b}\}
\]
and
\[
|A_{1}(x,y)|\leq K(1+|y|^{q_{A_{1}}}), \text{ with }q_{A_{1}}=q_{\sigma}+q_{b}
\]

By Lemma \ref{L:productBound} we then get that the desired bounds due the restrictions of Condition \ref{C:Tightness}, i.e., $q_{A_{1}}<1$ and $q_{B_{1}}<2$, hold. Namely, we obtain that
\begin{align*}
\sup_{\eps\in(0,\eps_{0})}\sum_{j=2}^{5}\mathbb{E}\left(\sup_{t\in[0,T]}| \hat{\eta}^{1,j,\eps}_{t}|\right)
 & \leq C
\end{align*}
for some constant $C<\infty$. As far as the stochastic integral terms $\hat{\eta}^{1,j,\eps}_{t}$ for $j=6, 7$ we proceed along similar lines as follows. The hardest term to treat is  the first component of $\hat{\eta}^{1,6,\eps}_{t}$. We have  that for a constant $C<\infty$ that may change from line to line and  for  $q_{\nabla_{x}\chi \sigma}<1$ (where $q_{\nabla_{x}\chi \sigma}$ denotes the degree of polynomial growth in $|y|$ of the norm of  $\left( \nabla_{x}\chi \right) (x,y)\sigma(x,y)$), we have for some constant $C<\infty$ that may change from inequality to inequality
\begin{align*}
  \mathbb{E}\left[\sup_{t\in[0,T]} \left\lvert \int_{0}^{t} \left( \nabla_{x}\chi \right) \sigma(\hat{X}_s^{\eps,v^\eps}, \hat{Y}_s^{\eps,v^\eps}) d W_{s} \right\rvert^{2}  \right]  &\leq C
    \mathbb{E} \int_{0}^{T} \left| \left( \nabla_{x}\chi \right) \sigma(\hat{X}_s^{\eps,v^\eps}, \hat{Y}_s^{\eps,v^\eps})\right|^{2} d s < C,
\end{align*}

from which the result follows by Lemma~\ref{L:productBound} given that $q_{\nabla_{x}\chi \sigma}=q_{\sigma}+q_{b}<1$. The other stochastic integral terms are treated along the same lines. Hence, we get
\begin{align*}
\sup_{\eps\in(0,\eps_{0})}\sum_{j=6}^{7}\mathbb{E}\left(\sup_{t\in[0,T]}| \hat{\eta}^{1,j,\eps}_{t}|\right)
 & \leq C.
\end{align*}

Thus overall we have obtained that for some $\eps_{0}>0$ and for a constant $C<\infty$
\begin{align}
\sup_{\eps\in(0,\eps_{0})}\mathbb{E}\left(\sup_{t\in[0,T]}| \hat{\eta}^{1,\eps}_{t}|\right)
 & \leq C. \label{Eq:CompactContainment1}
\end{align}

Next, let us treat the second term in \eqref{Eq:ControlledProcesses}, i.e, the term
\[
\hat{\eta}^{2,\eps}_{t}=\int_{t_0}^{t} \frac{\lambda_1(\hat{X}_s^{\eps, v^{\eps}}, \hat{Y}_s^{\eps, v^{\eps}}) - \bar{\lambda}_1(\hat{X}_s^{\eps, v^{\eps}})}{\sqrt{\eps} h(\eps)} ds.
\]

To do so, we apply the It\^{o} formula to the solution $\Phi_{1}$ of \eqref{E:Phi}. After rearranging terms, we get
\begin{align} \label{E:lambdaExpand}
  &\frac{1}{\sqrt{\eps} h(\eps)} \int_0^t \left( \lambda_1( \hat{X}_s^{\eps, v^\eps}, \hat{Y}_s^{\eps, v^\eps} ) - \bar{\lambda}_1(\hat{X}_s^{\eps, v^\eps}) \right) \,ds = - \frac{\delta^2/ \eps}{\sqrt{\eps} h(\eps)} \left( \Phi_1(\hat{X}_t^{\eps, v^\eps}, \hat{Y}_t^{\eps, v^\eps}) - \Phi_1 ( x_0, y_0 ) \right) \\
  & + \frac{\delta^2 / \eps}{\sqrt{\eps} h(\eps)}  \int_0^t \left( \nabla_x \Phi_1 \right) ( \hat{X}_s^{\eps, v^\eps}, \hat{Y}_s^{\eps, v^\eps} ) \left( \frac{\eps}{\delta} b( \hat{X}_s^{\eps, v^\eps}, \hat{Y}_s^{\eps, v^\eps} ) + c( \hat{X}_s^{\eps, v^\eps}, \hat{Y}_s^{\eps, v^\eps} ) \right) \,ds \notag \\
  & + \frac{\delta^2 / \eps}{\sqrt{\eps} h(\eps)}  \int_0^t \frac{\eps}{2} \sigma\sigma^\T( \hat{X}_s^{\eps, v^\eps}, \hat{Y}_s^{\eps, v^\eps} ) : \nabla_x\nabla_x \Phi_1( \hat{X}_s^{\eps, v^\eps}, \hat{Y}_s^{\eps, v^\eps} ) \,ds \notag \\
  & +  \frac{\delta^2 / \eps}{\sqrt{\eps} h(\eps)}  \int_0^t \frac{\eps}{\delta} \sum_{i,k}\partial_{y_{i}}\partial_{x_{k}}   \Phi_1( \hat{X}_s^{\eps, v^\eps}, \hat{Y}_s^{\eps, v^\eps} ) (\sigma\tau^{\T}_{1})_{i,k}( \hat{X}_s^{\eps, v^\eps}, \hat{Y}_s^{\eps, v^\eps} )  \,ds \notag \\
   & + \frac{\delta^2}{\eps} \int_0^t \left( \nabla_x \Phi_1 \right) ( \hat{X}_s^{\eps, v^\eps}, \hat{Y}_s^{\eps, v^\eps} ) \sigma( \hat{X}_s^{\eps, v^\eps}, \hat{Y}_s^{\eps, v^\eps} ) (u_{1}(s)-v_{1}^\eps(s)) \,ds \notag \\
    &+ \frac{\delta / \eps}{ \sqrt{\eps} h(\eps)} \int_0^t \left( \nabla_y \Phi_1 \right) ( \hat{X}_s^{\eps, v^\eps}, \hat{Y}_s^{\eps, v^\eps} ) g( \hat{X}_s^{\eps, v^\eps}, \hat{Y}_s^{\eps, v^\eps} ) \,ds \notag \\
   & + \frac{\delta}{\eps} \int_0^t \left( \nabla_y \Phi_1 \right) ( \hat{X}_s^{\eps, v^\eps}, \hat{Y}_s^{\eps, v^\eps} ) \left[\tau_1( \hat{X}_s^{\eps, v^\eps}, \hat{Y}_s^{\eps, v^\eps} ) (u_{1}(s)-v_{1}^\eps(s)) + \tau_2( \hat{X}_s^{\eps, v^\eps}, \hat{Y}_s^{\eps, v^\eps} ) (u_{2}(s)-v_{2}^\eps(s)) \right] \,ds \notag \\
  & + \frac{\delta^2}{\eps h(\eps)} \int_0^t \left( \nabla_x \Phi_1 \right) ( \hat{X}_s^{\eps, v^\eps}, \hat{Y}_s^{\eps, v^\eps} ) \sigma( \hat{X}_s^{\eps, v^\eps}, \hat{Y}_s^{\eps, v^\eps} ) \,dW_s \notag \\
  & + \frac{\delta}{\eps h(\eps)} \int_0^t \left( \nabla_y \Phi_1 \right) ( \hat{X}_s^{\eps, v^\eps}, \hat{Y}_s^{\eps, v^\eps} ) \tau_1( \hat{X}_s^{\eps, v^\eps}, \hat{Y}_s^{\eps, v^\eps} ) \,dW_s \notag \\
  & + \frac{\delta}{\eps h(\eps)} \int_0^t \left( \nabla_y \Phi_1 \right) ( \hat{X}_s^{\eps, v^\eps}, \hat{Y}_s^{\eps, v^\eps} ) \tau_2( \hat{X}_s^{\eps, v^\eps}, \hat{Y}_s^{\eps, v^\eps} ) \,dB_s . \notag\\
   &=\sum_{j=1}^{10}\hat{\eta}^{2,j,\eps}_{t}.\notag
\end{align}

From this representation, we obtain the result that we want, in exactly the same way as we did for the first term. In particular, $\hat{\eta}^{2,1,\eps}_{t}$ follows as $\hat{\eta}^{1,1,\eps}_{t}$ and the rest of the terms are Riemann and stochastic integral terms. For example, the Riemann integral terms (ignoring the prefactors involving $\eps$ and $\delta$; notice that all the prefactors go to zero) are of the form
\begin{align*}
\int_{0}^{t}B_{2}(\hat{X}_s^{\eps,v^\eps}, \hat{Y}_s^{\eps,v^\eps})ds \text{ or } \int_{0}^{t}A_{2}(\hat{X}_s^{\eps,v^\eps}, \hat{Y}_s^{\eps,v^\eps})v_{i}^{\eps}(s)ds
\end{align*}
for appropriate vector valued functions $B_{2}(x,y)$ and matrix valued functions $A_{2}(x,y)$ and $i=1,2$.
Now due to the growth assumption of Condition \ref{C:growth} and Theorem \ref{T:regularity} we notice that
\[
|B_{2}(x,y)|\leq K(1+|y|^{q_{B_{2}}}), \text{ with }q_{B_{2}}=\max\{2q_{\sigma}+(q_{b}\vee q_{c}),q_{\sigma}+(q_{b}\vee q_{c})+(q_{\sigma}\vee q_{b})\}=q_{\sigma}+(q_{b}\vee q_{c})+(q_{\sigma}\vee q_{b})
\]
and
\[
|A_{2}(x,y)|\leq K(1+|y|^{q_{A_{2}}}), \text{ with }q_{A_{2}}=q_{\sigma}+(q_{b}\vee q_{c}).
\]

In the end, using Lemma \ref{L:YIntegralgrowth} and Lemma \ref{L:productBound}, we obtain that for some $\eps_{0}>0$ and for a constant $C<\infty$
\begin{align}
\sup_{\eps\in(0,\eps_{0})}\mathbb{E}\left(\sup_{t\in[0,T]}| \hat{\eta}^{2,\eps}_{t}|\right)
 & \leq C. \label{Eq:CompactContainment2}
\end{align}

In regards to the third term in \eqref{Eq:ControlledProcesses}, i.e., to
\[
\hat{\eta}^{3,\eps}_{t}=\int_{t_0}^{t} \frac{\bar{\lambda}_1(\hat{X}_s^{\eps, v^{\eps}}) - \bar{\lambda}_1(\bar{X}_s)}{\sqrt{\eps} h(\eps)} ds
\]
we proceed as follows. Lipschitz continuity of the function $\bar{\lambda}_{1}$ gives
\begin{equation} \label{E:lambdabarBound}
  \sup_{0\le t\le T} \left\lvert \frac{1}{\sqrt{\eps} h(\eps)} \int_0^t \left( \bar{\lambda}_1(\hat{X}_s^{\eps, v^\eps}) - \bar{\lambda}_1(\bar{X}_s) \right) \,ds \right\rvert \le L_\lambda \int_0^T \left\lvert \hat{\eta}_{s}^{\eps, v^\eps} \right\rvert \,ds.
\end{equation}

Finally, we notice that terms $\hat{\eta}^{i,\eps}_{t}$ for $i=4,5,6$ in \eqref{Eq:ControlledProcesses} are simply Riemann and stochastic integral terms that can be treated the same way as the Riemann and stochastic integral terms of $\hat{\eta}^{1,\eps}_{t}$.  In short, we obtain
that for some $\eps_{0}>0$ and for a constant $C<\infty$
\begin{align}
\sum_{i=4}^{6}\sup_{\eps\in(0,\eps_{0})}\mathbb{E}\left(\sup_{t\in[0,T]}| \hat{\eta}^{i,\eps}_{t}|\right)
 & \leq C. \label{Eq:CompactContainment3}
\end{align}

Now putting together the estimates \eqref{Eq:CompactContainment1}, \eqref{Eq:CompactContainment2}, \eqref{E:lambdabarBound}, and \eqref{Eq:CompactContainment3} we obtain
\begin{equation*}
  \E \sup_{t\in[0,T]}\lvert \hat{\eta}_t^{\eps, v^\eps} \rvert \le C + L_\lambda \int_0^T \E \sup_{s\in[0,t]} \left\lvert \hat{\eta}_{s}^{\eps, v^\eps} \right\rvert \,dt
\end{equation*}
where $C$ is the sum of the upper bounds on the expectations of the terms $\mathbb{E}\left(\sup_{t\in[0,T]}| \hat{\eta}^{i,\eps}_{t}|\right)$ for $i=1,2,4,5,6$.  Then by Gronwall's lemma, we have that there is some $\eps>0$ and some constant $C<\infty$ such that
\begin{equation} \label{E:etaBound}
  \E \sup_{t\in[0,T]}\lvert \hat{\eta}_t^{\eps, v^\eps} \rvert \le C.
\end{equation}

The latter statement now immediately implies the first statement of the proposition, i.e.\ \eqref{Eq:CompactContainment}, using Markov's inequality.

The second statement~\eqref{Eq:ContinuousReg} follows by similar arguments using part (iii) of Lemma~\ref{L:productBound} and the growth properties of the involved functions with respect to~$|y|$. The only exception to this are the terms $\hat{\eta}^{1,1,\eps}_{t}$ and $\hat{\eta}^{2,1,\eps}_{t}$. For these terms we need to show that 
 there exists some $\eps_{0}>0$ such that for every  $\zeta_{2}>0$, there exists  $\zeta_{1}>0$ with the property
\[
\sup_{\eps\in(0,\eps_{0})}\mathbb{P}\left[\sup_{t\in[0,T]}|\hat{\eta}^{i,1,\eps}_{t}|>\zeta_{1}\right]\leq \zeta_{2}, \text{ for }i=1,2.
\]

For $i=1$, this follows from estimate \eqref{Eq:Eta_1_1term} on $\mathbb{E}\left(\sup_{t\in[0,T]}| \hat{\eta}^{1,1,\eps}_{t}|\right)$ together with  Markov's inequality (the statement for $i=2$ is basically identical).

This completes the proof of the proposition.
\end{proof}

\section{On relaxing the growth properties of the subsolution}\label{S:RelaxedConditions}

In this section, we discuss the possibility of relaxing the conditions on the growth of the subsolution $\check{U}(s,\eta)$ on $\eta$. Recall that in Condition \ref{Cond:ExtraReg} we assume that the first derivative of $\check{U}$ is bounded uniformly with respect to $\eta$.  In this section, we investigate whether it is possible to relax this. It turns out that even though this is possible, it depends on the growth of the coefficients $b,c,\sigma$, i.e., on $q_{b},q_{c},q_{\sigma}$.

Let us replace Conditions \ref{Cond:ExtraReg} and \ref{C:Tightness} by Condition \ref{Cond:ExtraReg2} below.
\begin{condition}
\label{Cond:ExtraReg2} There exists a subsolution $\check{U}$ which has continuous derivatives up to
order $1$ in $s$ and order $2$ in $\eta$. In addition, there exists a constant $0<C<\infty$ and a constant $q_{\check{U}}$, such that for $\left((q_{b}\vee q_{c})+q_{\sigma}+(q_{\sigma}\vee q_{b})\right)<1$ we have
\[
0\leq q_{\check{U}}\leq1-\left((q_{b}\vee q_{c})+q_{\sigma}+(q_{\sigma}\vee q_{b})\right)
\]
 and
\[
\sup_{s\in[t_0,T]}\left|\nabla_{\eta}\check{U}(s,\eta)\right|\leq C(1+|\eta|^{q_{\check{U}}}).
\]
\end{condition}

\begin{remark}
Notice that Condition \ref{Cond:ExtraReg2} implies for example that in the case where all the coefficients are uniformly bounded, then one can assume quadratic growth of the subsolution $\check{U}(s,\eta)$ with respect to $\eta$ with all the theoretical results of this paper remaining valid. Notice that in the examples of Section \ref{S:Simulations}, the subsolutions used for the importance sampling change of measure indeed grow quadratically on $\eta$.
\end{remark}
The changes that occur are in the proofs of Lemma \ref{L:YIntegralgrowth} and Proposition \ref{P:tightness}. In particular, now we have to deal with upper bounds of the form $\int_{t_0}^{T}|\hat{Y}_{s}^{\eps}|^{\nu_{1}}|\hat{\eta}_{s}^{\eps}|^{\nu_{2}}ds$ for appropriate $\nu_{1}>0$ and $\nu_{2}>0$. In the case of Condition \ref{Cond:ExtraReg}, we always had $\nu_{2}=0$. We will not repeat here the lengthy calculations (because they essentially follow via the same steps albeit with more tedious algebra), but we state below statements of the results as well as go over the changes required for the sake of completeness.

In the results that follow, we have set  $q_{1}=q_{\sigma}\vee q_{b}$ in the case of Regime 1 and $q_{1}=q_{\sigma}\vee q_{b}\vee q_{c}$ in the case of Regime 2. We also set $t_{0}=0$ for notational convenience.

\begin{lemma} \label{L:YIntegralgrowthRelaxed1}
Assume that Conditions~\ref{C:growth}, \ref{C:ergodic}, \ref{C:center} and \ref{Cond:ExtraReg2} hold. Consider any family $\{v^{\eps},\eps>0\}$ of
controls in $\mathcal{A}$ satisfying, for some $R<\infty$,
\begin{equation*}
    \sup_{\eps > 0} \int_{0}^T | v^{\eps}(s) |^2 d s < R
\end{equation*}
almost surely. Then there exist $\eps_{0}>0$ small enough such that
 \begin{equation}
    \mathbb{E}  \int_{0}^{T} | \hat{Y}_s^{\eps, v^{\eps}}|^{2} d s  \le K(R,T)\left(1+\frac{\delta^{2}}{\eps}h(\eps)^{2}\mathbb{E}\int_{0}^{T}| \hat{\eta}_s^{\eps, v^{\eps}}|^{\frac{2q_{\check{U}}}{1-q_{1}}} ds\right),\label{Eq:YboundGeneralBound1}
\end{equation}
and
\begin{equation}
\frac{\delta^{2}}{\eps}\mathbb{E}\left(\sup_{t\in[0,T]}\left|\hat{Y}_s^{\eps, v^{\eps}}\right|^{2}\right)\leq K(R,T)\left(1+\frac{\delta}{\sqrt{\eps}}h(\eps)\left(1+\mathbb{E}\left(\sup_{t\in[0,T]}\left|\hat{\eta}_s^{\eps, v^{\eps}}\right|^{\frac{2q_{\check{U}}}{1- q_{1}}}\right)\right)\right), \label{Eq:YboundGeneralBound2}
\end{equation}
for some finite constant $K(R,T)$ that may depend on $(R,T)$, but not on $\eps,\delta$.
\end{lemma}
\begin{proof}[Proof of Lemma \ref{L:YIntegralgrowthRelaxed1}]
By carefully following the proof of Lemma \ref{L:YIntegralgrowth} we see that \eqref{Eq:BoundL2normY} now takes the form
\begin{align}
\mathbb{E}\int_{0}^{T}| \hat{Y}_s^\eps|^{2}dt&\leq  C_{0}\left[1+ \frac{\delta^{2}}{\eps}|y_{0}|^{2} +N \frac{\delta^{2}}{\eps}h^{2}(\eps) \mathbb{E}\sup_{s\leq T}\left\Vert\tau(\hat{X}_s^\eps, \hat{Y}_s^\eps)\right\Vert^{2}+  \frac{\delta^{2}}{\eps}h^{2}(\eps) \mathbb{E}\sup_{s\leq T}\left\Vert\tau(\hat{X}_s^\eps, \hat{Y}_s^\eps)\right\Vert^{2}\int_{0}^{T}\left|\hat{Y}_s^\eps\right|^{2q_{1}}\left|\hat{\eta}_s^\eps\right|^{2q_{\check{U}}}ds\right.\nonumber\\
&\qquad \qquad \left.+
\int_{0}^{T} \mathbb{E}\left\Vert\tau(\hat{X}_s^\eps, \hat{Y}_s^\eps)\right\Vert^{2} ds\right].\nonumber
\end{align}

Then for $\eps,\delta$ small enough and after applying Young's inequality to the term $\int_{0}^{T}\left|\hat{Y}_s^\eps\right|^{2q_{1}}\left|\hat{\eta}_s^\eps\right|^{2q_{\check{U}}}ds$ we obtain directly \eqref{Eq:YboundGeneralBound1}. The derivation of \eqref{Eq:YboundGeneralBound2} is similar.
\end{proof}

Essentially \eqref{Eq:YboundGeneralBound1} and \eqref{Eq:YboundGeneralBound2} mean that the indicated upper bounds are now coupled with the behavior of the moderate deviations process $\hat{\eta}^{\eps, v^{\eps}}$, whereas before they were independent from it. Then, Condition \ref{Cond:ExtraReg2} together with Lemma \ref{L:YIntegralgrowthRelaxed1} allow us to get an a-priori bound for $\mathbb{E}\left(\sup_{t\in[0,T]}| \hat{\eta}^{\eps}_{t}|^{2}\right)$. Notice that, in the previous section we had assumed $q_{\check{U}}=0$ and we derived a bound for $\mathbb{E}\left(\sup_{t\in[0,T]}| \hat{\eta}^{\eps}_{t}|\right)$, see (\ref{E:etaBound}). In this section, due to allowing $q_{\check{U}}>0$ we strengthen the bound to $\mathbb{E}\left(\sup_{t\in[0,T]}| \hat{\eta}^{\eps}_{t}|^{2}\right)$. This is due to the coupling between the bounds for appropriate norms of $\hat{Y}_s^{\eps, v^{\eps}}$  and the corresponding norms for $\hat{\eta}_s^{\eps, v^{\eps}}$ as stated in Lemma \ref{L:YIntegralgrowthRelaxed1}. There was no such coupling in the setting of the previous section.  We have the following proposition.

\begin{proposition}\label{P:bound_eta_relaxed}
Assume that Conditions~\ref{C:growth}, \ref{C:ergodic}, \ref{C:center} and \ref{Cond:ExtraReg2} hold. Consider any family $\{v^{\eps},\eps>0\}$ of
controls in $\mathcal{A}$ satisfying, for some $R<\infty$,
\begin{equation*}
    \sup_{\eps > 0} \int_{0}^T | v^{\eps}(s) |^2 d s < R
\end{equation*}
almost surely. Then there exist $\eps_{0}>0$ small enough and a constant $C<\infty$ such that
 \begin{equation*}
    \sup_{\eps\in(0,\eps_{0})}\E \sup_{t\in[0,T]}\lvert \hat{\eta}_t^{\eps, v^\eps} \rvert^{2} \le C.
\end{equation*}
\end{proposition}

\begin{proof}
The proof follows closely the proof of $\mathbb{E}\sup_{t\in[0,T]}| \hat{\eta}^{\eps}_{t}|$ that was derived within the proof of Proposition \ref{P:tightness}. Below, we keep the notation used in the proof of Proposition \ref{P:tightness} and we outline the main differences. As before, we give the proof for Regime 1, and the proof for Regime 2 is nearly identical.

The first main difference comes in the treatment of the terms in the expressions
 $\hat{\eta}_t^{1,1,\eps}$ and $\hat{\eta}_t^{2,1,\eps}$. Let us explain the first one. Using \eqref{Eq:YboundGeneralBound2}, we have
\begin{align*}
\mathbb{E}\left(\sup_{t\in[0,T]}| \hat{\eta}^{1,1,\eps}_{t}|^{2}\right)
 & \leq C \frac{\delta^{2}}{\eps h(\eps)^{2}}\left(1+\mathbb{E}\sup_{t\in[0,T]}| Y^{\eps,v^{\eps}}_{t}|^{2 q_{b}}\right)
 \leq C\frac{\delta^{2}}{\eps h(\eps)^{2}}\left(1+\left(\mathbb{E}\sup_{t\in[0,T]}| Y^{\eps,v^{\eps}}_{t}|^{2}\right)^{q_{b}}\right)\nonumber\\
 &\leq C \frac{\delta^{2}}{\eps h(\eps)^{2}}\left(1+\frac{\eps}{\delta^{2}}+\frac{\sqrt{\eps}}{\delta}h(\eps)\left(1+\mathbb{E}\left(\sup_{t\in[0,T]}\left|\hat{\eta}_s^{\eps, v^{\eps}}\right|^{\frac{2q_{\check{U}}}{1- q_{1}}}\right)\right)\right)\nonumber\\
 &\leq C \left(\frac{\delta^{2}}{\eps h(\eps)^{2}}+\frac{1}{h(\eps)^{2}}+\frac{\delta}{\sqrt{\eps}h(\eps)}\left(1+\mathbb{E}\left(\sup_{t\in[0,T]}\left|\hat{\eta}_s^{\eps, v^{\eps}}\right|^{\frac{2q_{\check{U}}}{1- q_{1}}}\right)\right)\right).\nonumber
\end{align*}

The second main difference comes in the treatment of the terms in the expressions for (ignoring prefactors of order one or that go to zero as $\eps,\delta\downarrow 0$)
\[
\hat{\eta}^{1,3,\eps}_{t}=\int_0^t \left( \nabla_x \chi \right) (\hat{X}_s^{\eps,v^\eps}, \hat{Y}_s^{\eps,v^\eps}) \sigma(\hat{X}_s^{\eps,v^\eps}, \hat{Y}_s^{\eps,v^\eps}) u_1^{\eps}(s) \,ds
\]
in \eqref{E:coefsExpand}, and
\[
\hat{\eta}^{2,4,\eps}_{t}=\int_0^t \left( \nabla_x \Phi_{1} \right) (\hat{X}_s^{\eps,v^\eps}, \hat{Y}_s^{\eps,v^\eps}) \sigma(\hat{X}_s^{\eps,v^\eps}, \hat{Y}_s^{\eps,v^\eps}) u_1^{\eps}(s) \,ds
\]
in \eqref{E:lambdaExpand}. Note that both of these terms involve the subsolution. Let's study the second term which is also the more cumbersome one. Recalling from \eqref{Eq:feedback_controlReg1} that
\begin{equation*}
u_{1}^{\eps}(s)=u_{1}(s,\hat{\eta}_s^{\eps,v^\eps},\hat{Y}_s^{\eps,v^\eps})=-\alpha_{1}^{\T}(\bar{X}_{s},\hat{Y}_s^{\eps,v^\eps})\nabla_{\eta}\check{U}(s,\hat{\eta}_s^{\eps,v^\eps}),
\end{equation*}
 and setting for notational convenience $q_{Y}=(q_{b}\vee q_{c})+q_{\sigma}+(q_{\sigma}\vee q_{b})<1$, we obtain for some unimportant constant $C<\infty$
\begin{align*}
&\mathbb{E}\sup_{t\in[0,T]}\left(\int_0^t \left( \nabla_x \Phi_{1} \right) (\hat{X}_s^{\eps,v^\eps}, \hat{Y}_s^{\eps,v^\eps}) \sigma(\hat{X}_s^{\eps,v^\eps}, \hat{Y}_s^{\eps,v^\eps}) u_1^{\eps}(s) \,ds\right)^{2}\leq
C\mathbb{E}\int_{0}^{T}\left(1+|\hat{Y}_s^{\eps,v^\eps}|^{2q_{Y}}\right)\left(1+|\hat{\eta}_s^{\eps,v^\eps}|^{2q_{\check{U}}}\right)ds\nonumber\\
&\qquad\leq
C\mathbb{E}\int_{0}^{T}\left(1+|\hat{Y}_s^{\eps,v^\eps}|^{2q_{Y}}+|\hat{\eta}_s^{\eps,v^\eps}|^{2q_{\check{U}}}+|\hat{Y}_s^{\eps,v^\eps}|^{2q_{Y}}|\hat{\eta}_s^{\eps,v^\eps}|^{2q_{\check{U}}}\right)ds\nonumber\\
&\qquad\leq
C\mathbb{E}\int_{0}^{T}\left(1+|\hat{Y}_s^{\eps,v^\eps}|^{2q_{Y}}+|\hat{\eta}_s^{\eps,v^\eps}|^{2q_{\check{U}}}+|\hat{Y}_s^{\eps,v^\eps}|^{2}+|\hat{\eta}_s^{\eps,v^\eps}|^{2q_{\check{U}}\frac{1}{1-q_{Y}}}\right)ds\nonumber\\
&\qquad\leq
C\mathbb{E}\int_{0}^{T}\left(1+|\hat{Y}_s^{\eps,v^\eps}|^{2}+|\hat{\eta}_s^{\eps,v^\eps}|^{2q_{\check{U}}\frac{1}{1-q_{Y}}}\right)ds.\nonumber\\
\end{align*}

In the last inequality we applied the generalized Young's inequality $ab\leq \frac{1}{p}a^{p}+\frac{1}{q}b^{q}$ for $a,b\geq 0$, $1/p+1/q=1$ and $p=\frac{2}{2q_{Y}}=1 / q_{Y}>1$. Hence, using now \eqref{Eq:YboundGeneralBound1} we subsequently obtain for the last inequality for $\eps$ small enough
\begin{align*}
&\mathbb{E}\sup_{t\in[0,T]}\left(\int_0^t \left( \nabla_x \Phi_{1} \right) (\hat{X}_s^{\eps,v^\eps}, \hat{Y}_s^{\eps,v^\eps}) \sigma(\hat{X}_s^{\eps,v^\eps}, \hat{Y}_s^{\eps,v^\eps}) u_1^{\eps}(s) \,ds\right)^{2}\nonumber\\
&\qquad\leq
C\left(1+\frac{\delta^{2}}{\eps}h(\eps)^{2}\mathbb{E}\int_{0}^{T}| \hat{\eta}_s^{\eps, v^{\eps}}|^{\frac{2q_{\check{U}}}{1-q_{b}\vee q_{\sigma}}} ds+ \mathbb{E}\int_{0}^{T}|\hat{\eta}_s^{\eps,v^\eps}|^{2q_{\check{U}}\frac{1}{1-q_{Y}}}ds\right)\leq C\left(1+ \mathbb{E}\int_{0}^{T}|\hat{\eta}_s^{\eps,v^\eps}|^{2q_{\check{U}}\frac{1}{1-q_{Y}}}ds\right).\nonumber
\end{align*}

Doing calculations along the same lines for the rest of the terms (similarly to the proof of Proposition \ref{P:tightness}), in the end we obtain using \eqref{Eq:YboundGeneralBound1} for a constant $C<\infty$ that may change from line to line and for $\eps$ small enough
\begin{align*}
&\E \sup_{t\in[0,T]}\lvert \hat{\eta}_t^{\eps, v^\eps} \rvert^{2} \le C\left[\frac{\delta^{2}}{\eps h(\eps)^{2}}+\frac{1}{h(\eps)^{2}}+\frac{\delta}{\sqrt{\eps}h(\eps)}\left(1+\mathbb{E}\left(\sup_{t\in[0,T]}\left|\hat{\eta}_s^{\eps, v^{\eps}}\right|^{\frac{2q_{\check{U}}}{1- q_{1}}}\right)\right) + \E\int_0^T   \left\lvert \hat{Y}_{s}^{\eps, v^\eps} \right\rvert^{2}ds\right.\nonumber\\
 &\qquad\qquad\left.+\int_0^T \E \sup_{s\in[0,t]} \left\lvert \hat{\eta}_{s}^{\eps, v^\eps} \right\rvert^{2}ds+ \int_0^T \E \sup_{s\in[0,t]} \left\lvert \hat{\eta}_{s}^{\eps, v^\eps} \right\rvert^{\frac{2q_{\check{U}}}{1-q_{Y}}} \,ds\right]\nonumber\\
&\qquad\qquad\le C\left[\frac{\delta^{2}}{\eps h(\eps)^{2}}+\frac{1}{h(\eps)^{2}}+\frac{\delta}{\sqrt{\eps}h(\eps)}\left(1+\mathbb{E}\left(\sup_{t\in[0,T]}\left|\hat{\eta}_s^{\eps, v^{\eps}}\right|^{\frac{2q_{\check{U}}}{1- q_{1}}}\right)\right)  \right.\nonumber\\
 &\qquad\qquad\left.+\frac{\delta^{2}}{\eps}h(\eps)^{2}\mathbb{E}\int_{0}^{T}| \hat{\eta}_s^{\eps, v^{\eps}}|^{\frac{2q_{\check{U}}}{1-q_{1}}} d s+\int_0^T \E \sup_{s\in[0,t]} \left\lvert \hat{\eta}_{s}^{\eps, v^\eps} \right\rvert^{2}ds+ \int_0^T \E \sup_{s\in[0,t]} \left\lvert \hat{\eta}_{s}^{\eps, v^\eps} \right\rvert^{\frac{2q_{\check{U}}}{1-q_{Y}}} \,ds\right].\nonumber\\
 &\qquad\qquad\le C\left[\frac{\delta^{2}}{\eps h(\eps)^{2}}+\frac{1}{h(\eps)^{2}}+\frac{\delta}{\sqrt{\eps}h(\eps)} +\frac{1}{2C}\mathbb{E}\left(\sup_{t\in[0,T]}\left|\hat{\eta}_s^{\eps, v^{\eps}}\right|^{\frac{2q_{\check{U}}}{1-q_{1}}}\right)  \right.\nonumber\\
 &\qquad\qquad\left.+\int_0^T \E \sup_{s\in[0,t]} \left\lvert \hat{\eta}_{s}^{\eps, v^\eps} \right\rvert^{2}ds+ \int_0^T \E \sup_{s\in[0,t]} \left\lvert \hat{\eta}_{s}^{\eps, v^\eps} \right\rvert^{\frac{2q_{\check{U}}}{1-q_{Y}}} \,ds\right].\nonumber
\end{align*}

In the last inequality we used the property  $\lim_{\eps\rightarrow 0}\frac{\delta}{\sqrt{\eps}h(\eps)}=0$. Hence for  sufficiently small $\eps>0$ and since by Condition \ref{Cond:ExtraReg2}, $0\leq q_{\check{U}}\leq 1-q_{Y}\leq 1-q_{1}\leq 1$, we obtain, using Gronwall's lemma   that for some small enough $\eps_{0}>0$ there is a constant $0<C<\infty$ such that
\begin{align*}
\sup_{\eps\in(0,\eps_{0})}\E \sup_{t\in[0,T]}\lvert \hat{\eta}_t^{\eps, v^\eps} \rvert^{2} &\le C.
\end{align*}
concluding the proof of the proposition.
\end{proof}

Then combining Lemma \ref{L:YIntegralgrowthRelaxed1} and Proposition \ref{P:bound_eta_relaxed} we get the following Lemma, i.e., we recover the statement of Lemma \ref{L:YIntegralgrowth}.
\begin{lemma} \label{L:YIntegralgrowthRelaxed2}
Assume that Conditions~\ref{C:growth}, \ref{C:ergodic}, \ref{C:center} and \ref{Cond:ExtraReg2} hold. Consider any family $\{v^{\eps},\eps>0\}$ of
controls in $\mathcal{A}$ satisfying, for some $R<\infty$,
\begin{equation*}
    \sup_{\eps > 0} \int_{t_0}^T | v^{\eps}(s) |^2 d s < R
\end{equation*}
almost surely. Then there exist $\eps_{0}>0$ small enough such that
 \begin{equation*}
    \mathbb{E}  \int_{t_0}^{T} | \hat{Y}_s^{\eps, v^{\eps}}|^{2} d s  \le K(R,T),
\end{equation*}
and
\begin{equation*}
\frac{\delta^{2}}{\eps}\mathbb{E}\left(\sup_{t\in[t_{0},T]}\left|\hat{Y}_s^{\eps, v^{\eps}}\right|^{2}\right)\leq K(R,T)\left(1+\frac{\delta}{\sqrt{\eps}}h(\eps)\right)
\end{equation*}
for some finite constant $K(R,T)$ that may depend on $(R,T)$, but not on $\eps,\delta$.
\end{lemma}

Then using Proposition \ref{P:bound_eta_relaxed} and Lemma \ref{L:YIntegralgrowthRelaxed2}, tightness  of the family $\{\hat{\eta}^{\eps},\eps>0\}$ on $\mathcal{C}([t_0,T];\mathbb{R}^{n})$ follows as in Proposition \ref{P:tightness} and the proof of Theorem \ref{T:UniformlyLogEfficient} goes through. Details are omitted.

\section{Simulation studies}\label{S:Simulations}

In this section we present some numerical studies in order to illustrate the theoretical results of this paper. 
Before presenting the numerical studies, let us first introduce some notation. The measure to compare the different estimators is the relative error of the estimator per sample.
In order to distinguish among the different Monte Carlo procedures, we will denote by $\rho^{\eps}_{NMC}$, $\rho^{\eps}_{LD}$  and $\rho^{\eps}_{MD}$  the relative error per sample for the naive Monte Carlo (i.e. no change of measure), for the large deviations--based importance sampling estimator and for the moderate deviations--based importance sampling estimator respectively. Analogously, let $\hat{\theta}_{NMC}(\eps)$, $\hat{\theta}_{LD}(\eps)$ and $\hat{\theta}_{MD}(\eps)$ be the corresponding estimators.

Let $\textrm{Var}^{\eps}\doteq\textrm{Var}(\hat{\theta}(\eps))$ be the variance of the estimator based on the change of measure induced by the appropriate control each time. The relative error of the estimator per sample based on
the change of measure induced by the corresponding control each time is defined as
\[
\mbox{relative error per sample} \doteq \hat{\rho}^{\eps}=\sqrt{N}\frac{\mbox{standard deviation of the estimator}}%
{\mbox{expected value of the estimator}}=\frac{\sqrt{\widehat{\textrm{Var}}^{\eps}}}{\hat{\theta}(\eps)}.
\]

The smaller the relative error per sample is, the more efficient the estimator is. However, in practice both the standard deviation and the expected value
of an estimator are typically unknown, which implies that empirical relative error is
often used for measurement. This means that the expected value of
the estimator will be replaced by the empirical sample mean, and the
standard deviation of the estimator will be replaced by the
empirical sample standard error.

In Section \ref{SS:TwoScaleSystem_a} we consider a system of slow-fast diffusions and we estimate functionals associated with rare events in the moderate deviations regime in parallel to the theory developed in this paper. In  Section \ref{SS:TwoScaleSystem_b} we switch gears slightly and even though we continue to consider the same model as in Section \ref{SS:TwoScaleSystem_a}, we are now interested in estimating rare events in the large deviations regime, but using the moderate deviation methodology. In this example, one cannot apply the LD--based IS  methodology directly as the corresponding HJB does not seem to provide, at least in an obvious way, subsolutions in closed form. On the other hand the moderate deviations does so, making its application quite straightforward. We conclude with Section \ref{SS:PeriodicProblem} where we consider diffusion in rough potentials and we look at an example where one can apply both LD--based IS and MD--based IS. We see that if the event is not too rare then the MD--based IS offers a viable alternative to the LD--based IS for multiscale problems.

\subsection{Example 1: A two-scale slow--fast system}\label{SS:TwoScaleSystem_a}

Consider the system of equations
\begin{align*}
dX^{\eps}_{t}&=-\partial_{x}V(X^{\eps}_{t}, Y^{\eps}_{t}) \,dt +\sqrt{\eps}\sqrt{2D}dW_{t}\nonumber\\
dY^{\eps}_{t}&=-\frac{\eps}{\delta^{2}}\partial_{y}V(X^{\eps}_{t}, Y^{\eps}_{t}) \,dt +\frac{\sqrt{\eps}}{\delta}dB_{t}\nonumber\\
(X^{\eps}_{0},Y^{\eps}_{0})&=(x_{0},y_{0}),
\end{align*}
where $W_{t}, B_{t}$ are standard independent one dimensional Brownian motions. otice that this is a standard slow-fast system with the time scale separation parameter being $\nu$ such that $1/\nu=\eps/\delta^{2}$, see for example \cite{BLP,PS}. Also, here we take
\[
V(x,y)=V_{1}(x)+V_{2}(x,y)
\]
where
\[
V_{1}(x)=\frac{1}{2}(x^{2}-1)^{2}, \text{ and }V_{2}(x,y)=\frac{1}{2}(x-y)^{2}.
\]

It is easy to see that in this case the corresponding invariant measure associated with the fast process $Y$ is the Gaussian measure
\[
\mu_{x}(dy)=\frac{1}{\sqrt{\pi}}e^{-(x-y)^{2}}.
\]

Hence, we obtain that $X^{\eps}_{t}\rightarrow \bar{X}_{t}$ in probability, where $\bar{X}_{t}$ satisfies the ordinary differential equation
\[
d\bar{X}_{t}=-2\bar{X}_{t}\left(\bar{X}^{2}_{t}-1\right) \,dt, \quad \bar{X}_{0}=x_{0}.
\]

It is easy to see that the dynamical system associated with $\bar{X}_{t}$ has two stable equilibria located at $-1$ and $1$ and an unstable one at $0$, with solutions converging exponentially fast to either $-1$ or to $1$ depending on which domain of attraction the initial point $x_{0}$ is.

Now let us set $H(\eta)=(\eta-3)^{2}$. We are interested in computing
\[
\theta(\eps)=\mathbb{E}\left[e^{-h^{2}(\eps)H(\eta^{\eps}_{T})}\right].
\]

Computation of $\theta(\eps)$ is associated to rare events if for example $x_{0}=-1$ and the setting falls in the scalings considered in this paper.

\subsubsection{Moderate deviations based scheme}

Let us now develop the moderate deviations importance sampling scheme. The MD related HJB equation  boils down to
\begin{eqnarray}
\partial_{t}G(t,\eta)- V_{1}^{''}(\bar{X}(t))\eta \partial_{\eta}G(t,\eta)-  D|\partial_{\eta}G(t,\eta)|^{2}&=&0
\nonumber\\
G(T,\eta)&=&H(\eta).\label{Eq:HJB_2_d_MD_a}
\end{eqnarray}

Now, in our case $V_{1}^{''}(x)=6x^{2}-2$ and our goal is to construct subsolutions to \eqref{Eq:HJB_2_d_MD_a}. Let us define
\[
c(t)=V_{1}^{''}(\bar{X}(t))=6|\bar{X}(t)|^{2}-2,
\]
and set $\gamma=3$. 


 We will be interested in rare events, which, in this case, means that the initial point is close to the stable equilibrium points of the limiting dynamics $\bar{X}_{t}$.  In particular, direct substitution shows that  if for example $x_{0}=-1$ then the function
\[
U(t,\eta)=\frac{(\gamma e^{4T}-\eta e^{4t})^2}{-\frac{D}{2} e^{8t}+(1+\frac{D}{2})e^{8T}}
\]
is an exact solution to \eqref{Eq:HJB_2_d_MD_a}. The reason is that if $x_{0}=-1$, then $\bar{X}(t)=-1$ for every $t$ (hence $c(t)=4$ for every $t\geq 0$).

 More generally, if $|x_{0}|>1/\sqrt{2}$ then
\begin{equation}
\check{U}(t,\eta)=\frac{(\gamma e^{\int_{0}^{T}c(s)ds}-\eta e^{\int_{0}^{t}c(s)ds})^2}{-2D e^{2\int_{0}^{t}c(s)ds}+(1+2D)e^{2\int_{0}^{T}c(s)ds}},
\label{Eq:Subsolution_2_D_a}
\end{equation}
is a subsolution to \eqref{Eq:HJB_2_d_MD_a} according to Definition \ref{Def:ClassicalSubsolution}. Indeed, by direct substitution, we may compute
\begin{align}
\partial_{t}\check{U}(t,\eta)- c(t)\eta \partial_{\eta}\check{U}(t,\eta)-  D|\partial_{\eta}\check{U}(t,\eta)|^{2}&=\frac{(\gamma e^{\int_{0}^{T}c(s)ds}-\eta e^{\int_{0}^{t}c(s)ds})^2}{\left(-2D e^{2\int_{0}^{t}c(s)ds}+(1+2D)e^{2\int_{0}^{T}c(s)ds}\right)^{2}}4D e^{\int_{0}^{t}c(s)ds}(c(t)-1).\nonumber
\end{align}

Noticing now that $c(t)-1=6(|\bar{X}(t)|^{2}-1/2)$, we obtain that because we assumed that $|x_{0}|>1/\sqrt{2}$ and because $\bar{X}(t)$ is attracted to either $-1$ or $1$ depending on the initial condition, then  $|\bar{X}(t)|^{2}>1/2$ for every $t\geq0$. This implies that $(c(t)-1)>0$ which then yields the validity of the subsolution property based on Definition \ref{Def:ClassicalSubsolution} (the inequality at the terminal time $T$ is clearly true by setting $t=T$ in (\ref{Eq:Subsolution_2_D_a})).

\subsubsection{Simulation Results for two--scale system}
Let us now summarize the results for the 2--d slow--fast problem described in this section.
We consider the case of Regime 2, in which case, by Theorem \ref{T:UniformlyLogEfficient}, the nearly optimal control is given by
$u(t,\eta,y)=\left(u_{1}(t,\eta,y),u_{2}(t,\eta,y)\right)$ with
\[
 u_{1}(t,\eta,y)=-\sqrt{2D}\partial_{\eta}U(t,\eta) \text{ and } u_{2}(t,\eta,y)=-\partial_{\eta}U(t,\eta).
\]

Below, we present results for the choice
\[
h(\eps)=\eps^{-0.45}.
\]

We used $N=2.5 \times 10^{6}$ trajectories with discretization step
\[
\text{T}_{\text{step}}=0.001\frac{\delta^{2}}{\eps}.
\]

In the simulation studies of Table \ref{Table4} below we chose as initial point $(x_{0},y_{0})=(-1,0)$.
\begin{table}[!ht]
\begin{center}
\begin{tabular}{|c|c|c|c|c|c|c|}
\hline
      $\eps$ & $\delta$&  $\hat{\theta}_{NMC}(\eps)$ &$\hat{\theta}_{MD}(\eps)$ &$\hat{\rho}^{\eps}_{NMC}$    & $\hat{\rho}^{\eps}_{MD}$     \\
     \hline  $0.5$ & $0.5$ & $1.33e-02 $  & $1.32e-02 $ &$ 4.74$ &  $1.73$ \\
    \hline  $0.3$ & $0.3$  &  $2.24e-03$  & $2.22e-03$ &$10.13$  & $2.64$ \\
   \hline  $0.1$ & $0.1$  & $1.64e-06$  & $1.68e-06$ &$205$  & $3.57$ \\
    \hline  $0.07$ & $0.07$   &$2.26e-08 $ & $2.66e-08 $ & $1035$  & $ 5.49$ \\
    \hline  $0.05$ & $0.05$  &$1.94e-13 $  & $1.33e-10 $ &$1569$  & $ 6.50$ \\
    \hline  $0.03$ & $0.03$  & $2.43e-36 $  & $1.05e-15 $ &$1587$&  $8.64$ \\
   \hline
\end{tabular}
\end{center}
\caption{Comparison table for slow-fast system in Regime 2.\label{Table4}}
\end{table}

Here $\hat{\theta}_{NMC}(\eps)$  and $\hat{\theta}_{MD}(\eps)$ are the point estimates based on the naive Monte Carlo and moderate deviations importance sampling scheme respectively. Analogously $\hat{\rho}^{\eps}_{NMC}$ and   $\hat{\rho}^{\eps}_{MD}$ are the relative error per sample for the naive Monte Carlo and moderate deviations importance sampling scheme respectively.

We see that if the event is not too rare then moderate deviations--based importance sampling will  work well in practice and is quite straightforward to apply here. Note that for small values of $\eps$, the naive Monte Carlo is no longer accurate as a consequence of the large relative error per sample.

\subsection{Example 2: The two-scale slow-fast system revisited}\label{SS:TwoScaleSystem_b}

Let us again consider the problem outlined in Section \ref{SS:TwoScaleSystem_a}, but consider now a rare event problem in the large deviations scaling.

We choose the cost function to be $R(x)=(x-1)^{2}$ and assume that we want to compute
\[
\theta(\eps)=\mathbb{E}\left[e^{-\frac{1}{\eps}R(X^{\eps}_{T})}\right]
\]

Computation of $\theta(\eps)$ is associated to rare events if for example $x_{0}=-1$. Indeed, in this case the function $R(x)$ is minimized at $x=1$ and for this to happen the dynamical system needs to go from one well of attraction to the other one.

\subsubsection{Moderate deviations based scheme}

Let us now develop the moderate deviations importance sampling scheme. For completeness let us first briefly discuss the situation in the large deviations scaling. With either the large deviations scaling or with the moderate deviations scaling one needs to be able to find subsolutions to the appropriate HJB equations. In the large deviations case, the appropriate HJB equation takes the form
\begin{align}
\partial_{t}G(t,x)- V_{1}^{\prime}(x) \partial_{x}G(t,x)-  D|\partial_{x}G(t,x)|^{2}&=0
\nonumber\\
G(T,x)&=R(x).\label{Eq:HJB_2d_systemLDP}
\end{align}

Solving \eqref{Eq:HJB_2d_systemLDP} requires numerical methods since the nonlinearity of $V_{1}^{\prime}(x)=2x(x^{2}-1)$ prohibits obtaining explicit solutions. Closed form subsolutions seem to be difficult to obtain as well. However, the situation is considerably easier in the moderate deviations regime.

To do so, we first need to re-express the event of interest in terms of the moderate deviations scaling. For this purpose, we have
\begin{align}
\mathbb{E}\left[e^{-\frac{1}{\eps}R(X^{\eps}(T))}\right]=\mathbb{E}\left[e^{-h^{2}(\eps)H(\eta^{\eps}(T);\beta)}\right],\label{Eq:MD_Estimation0}
\end{align}
where we have defined
\begin{eqnarray}
H(\eta;\beta)&=&
\left(\eta-\frac{1-\bar{X}(T)}{\beta}\right)^{2}, \text{ where }\beta=\sqrt{\eps}h(\eps),\nonumber
\end{eqnarray}
and
\[
\eta^{\eps}(t)=\frac{X^{\eps}(t)-\bar{X}(t)}{\sqrt{\eps}h(\eps)}.
\]

Now before proceeding, we need to comment on \eqref{Eq:MD_Estimation0}. Notice that the terminal condition that appears under the moderate deviations scaling on \eqref{Eq:MD_Estimation0} depends on $\eps$, i.e., $H(\eta;\sqrt{\eps}h(\eps))$. However, the theory that has been developed in this paper is for terminal conditions that are independent of $\eps$. This is an issue that naturally comes up in applications, and to the best of our knowledge it was first addressed in the recent works \cite{Johnson2015,DupuisJohnson2017}. Even though the setup of \cite{Johnson2015, DupuisJohnson2017} is different from ours, the discussion on this issue is essentially the same.

In every simulation problem of this sort, independently of whether it is large deviations or moderate deviations, we are dealing with a given, specific, value of $\eps$ which may or may not be sufficiently small. Then one does the simulation with the method of choice and with the specific given value of $\eps$ with the  expectation that the theory will be true to a certain degree at least. If we want to use the moderate deviations scaling, then inevitably (at least for the problem studied here) the cost function, $H(\eta)$ will depend on $\eps$ through the term $\beta=\sqrt{\eps} h(\eps)$. However, as is discussed in \cite{Johnson2015, DupuisJohnson2017}, and we also confirm via simulation here, the choice of the embedding for the value of $\beta$ does not influence the limiting logarithmic asymptotic. 

The MD related HJB equation  boils down to
\begin{eqnarray}
\partial_{t}G(t,\eta)- V_{1}^{''}(\bar{X}(t))\eta \partial_{\eta}G(t,\eta)-  D|\partial_{\eta}G(t,\eta)|^{2}&=&0
\nonumber\\
G(T,\eta)&=&H(\eta;\beta).\label{Eq:HJB_2_d_MD}
\end{eqnarray}

Now, in our case $V_{1}^{''}(x)=6x^{2}-2$ and our goal is to construct subsolutions to \eqref{Eq:HJB_2_d_MD}. Let us recall the definition $c(t)=V_{1}^{''}(\bar{X}(t))=6|\bar{X}(t)|^{2}-2$ and
set $\gamma=\frac{1-\bar{X}(T)}{\beta}$. With these definitions, consider now the function $\check{U}(t,\eta)$ defined in \eqref{Eq:Subsolution_2_D_a}, but with this new value for $\gamma$ now, which as we discussed before is  a subsolution to \eqref{Eq:HJB_2_d_MD} according to Definition \ref{Def:ClassicalSubsolution}.

%

By Theorem \ref{T:UniformlyLogEfficient}, the nearly optimal control is given by
\begin{itemize}
\item{Regime 1: ${u}(t,\eta,y)=\left({u}_{1}(t,\eta,y),{u}_{2}(t,\eta,y)\right)$ with
\[
 {u}_{1}(t,\eta,y)=-\sqrt{2D}\partial_{\eta}U(t,\eta) \text{ and } {u}_{2}(t,\eta,y)=0.
\]}
\item{Regime 2: ${u}(t,\eta,y)=\left({u}_{1}(t,\eta,y),{u}_{2}(t,\eta,y)\right)$ with
\[
 {u}_{1}(t,\eta,y)=-\sqrt{2D}\partial_{\eta}U(t,\eta) \text{ and } {u}_{2}(t,\eta,y)=-\partial_{\eta}U(t,\eta).
\]}
\end{itemize}

\subsubsection{Simulation Results for two-scale system}
Let us now summarize the results for the  problem described in this section.
Table \ref{Table2} has the  simulation results for this system in the case of Regime 1, whereas Table \ref{Table3} has the  simulation results for the system in the case of Regime 2.

Below, we present results for the choice
\[
h(\eps)=\eps^{-0.4}.
\]

For the sake of completeness, we mention here that the same simulations were also performed with $h(\eps)=\eps^{-0.5}$ (which is closer to the large deviations scaling), with $h(\eps)=\eps^{-0.1}$ and with $h(\eps)=\eps^{-0.2}$. In all of these cases, the results were statistically the same.

We used $N=2.5 \times 10^{6}$ trajectories with discretization step
\[
\text{T}_{\text{step}}=0.001\frac{\delta^{2}}{\eps}.
\]
In the simulation studies below we chose as initial point $(x_{0},y_{0})=(-1,0)$.
\begin{table}[!ht]
\begin{center}
\begin{tabular}{|c|c|c|c|c|c|c|c|c|}
\hline
      $\eps$ & $\delta$& $\eps/\delta$ & $j_{1}=\frac{\delta/\eps}{\sqrt{\eps}h(\eps)}$ & $\hat{\theta}_{NMC}(\eps)$ &$\hat{\theta}_{MD}(\eps)$ &$\hat{\rho}^{\eps}_{NMC}$    & $\hat{\rho}^{\eps}_{MD}$     \\
     \hline  $0.5$ & $0.3$ &  $1.67$& $0.64$& $8.85e-02 $  & $8.84e-02 $ &$ 2.59$ &  $0.97$ \\
    \hline  $0.25$ & $0.1$  & $2.5$ &$0.46$&  $1.18e-02$  & $1.18e-02$ &$7.30$  & $1.47$ \\
   \hline  $0.125$ & $0.04$  & $3.125$ & $0.39$& $2.66e-04$  & $2.57e-04$ &$45.71$  & $1.83$ \\
    \hline  $0.0625$ & $0.015$  & $4.17$ &$0.32$ &$1.10e-07 $ & $1.25e-07 $ & $1100$  & $ 2.86$ \\
    \hline  $0.03125$ & $0.0065$  & $4.81$ & $0.29$ &$3.11e-31 $  & $3.46e-14 $ &$1067$  & $ 5.53$ \\
    \hline  $0.025$ & $0.0045$  & $5.56$ &$0.26$& $2.75e-38 $  & $1.57e-17 $ &$1587$&  $13.94$ \\
 \hline
\end{tabular}
\end{center}
\caption{Comparison table for 2-d slow fast system in Regime 1.\label{Table2}}
\end{table}

\begin{table}[!ht]
\begin{center}
\begin{tabular}{|c|c|c|c|c|c|c|}
\hline
      $\eps$ & $\delta$&  $\hat{\theta}_{NMC}(\eps)$ &$\hat{\theta}_{MD}(\eps)$ &$\hat{\rho}^{\eps}_{NMC}$    & $\hat{\rho}^{\eps}_{MD}$     \\
     \hline  $0.5$ & $0.5$ & $1.04e-01 $  & $1.04e-01 $ &$ 2.35$ &  $1.52$ \\
    \hline  $0.25$ & $0.25$  &  $1.93e-02$  & $1.95e-02$ &$5.72$  & $1.17$ \\
   \hline  $0.125$ & $0.125$  & $1.06e-03$  & $1.09e-03$ &$23.67$  & $1.50$ \\
    \hline  $0.0625$ & $0.0625$   &$4.37e-06 $ & $4.91e-06 $ & $313$  & $ 2.45$ \\
    \hline  $0.03125$ & $0.03125$  &$9.57e-19 $  & $1.27e-10 $ &$1556$  & $ 6.34$ \\
    \hline  $0.025$ & $0.025$  & $4.67e-34 $  & $6.51e-13 $ &$1417$&  $9.14$ \\
       \hline  $0.015$ & $0.015$  & $0 $  & $1.62e-20 $ &$-$&  $15.38$ \\
 \hline
\end{tabular}
\end{center}
\caption{Comparison table for 2-d slow fast system in Regime 2.\label{Table3}}
\end{table}

We see that if the event is not too rare then moderate deviations--based importance sampling will  work well in practice and is quite straightforward to apply here. As in the previous example, we see that when $\eps$ becomes sufficiently small, the relative error for the naive Monte Carlo estimator grows until the estimator is no longer accurate. The large deviations counterpart would require numerically solving the related HJB equation, which in this case can of course be done, but it is computationally considerably more expensive than implementing  the moderate deviations--based scheme.

\subsection{Example 3 - Rare event simulation in rough potentials}\label{SS:PeriodicProblem}
Let us consider the following Langevin equation
\begin{equation}
dX^{\eps}_{t}=\left[ -\frac{\eps }{\delta }\nabla Q\left( \frac{%
X^{\eps }_{t}}{\delta }\right) -\nabla V\left( X^{\eps}_{t}\right) %
\right] dt+\sqrt{\eps }\sqrt{2D}dW_{t},\hspace{0.2cm}X^{\eps}_{0}=x_{0}. \nonumber
\end{equation}

Define the potential function to be
\begin{equation}
V(x)=\frac{1}{2}x^2, \qquad Q(y)=\cos(y)+\sin(y).\nonumber
\end{equation}

Notice that this is overdamped Langevin equation describing the motion of a Brownian particle in a rough potential, see for example \cite{Z88,DSW12}.

We choose the cost function $R(\cdot)$ to be
\begin{eqnarray}
R(x)&=&
\begin{cases}
(x-1)^2 & x\geq 0  \\
(x+1)^2 & x<0.
\end{cases}\nonumber
\end{eqnarray}

We want to compute
\[
\mathbb{E}\left[e^{-\frac{1}{\eps}R(X^{\eps}_{T})}\right].
\]

Now, this is a rare event problem because the function $e^{-\frac{1}{\eps}R(x)}$ is maximized at $x=\pm 1$, but in order for the process $X^{\eps}_{t}$ to hit the points $\pm 1$ a rare event has to take place.

This problem was studied in \cite{DSW12} using large deviations methods. This is a good example to compare large deviations--based importance sampling methods to moderate deviations--based importance sampling methods because for both cases one can compute appropriate subsolutions to the corresponding HJB problems. Therefore, one can compute the exact form of the change of measure for both cases. In Subsection \ref{LDbasedISperiodic} we review the large deviations--based importance sampling change of measure and in Subsection \ref{MDbasedISperiodic} we go over the moderate deviations--based importance sampling change of measure.

\subsubsection{Large deviations based scheme}\label{LDbasedISperiodic}
Define $\mathbb{T}$ to be the torus in dimension one with period $\lambda=2\pi$. Let us set
\[
L=\int_{\mathbb{T}}e^{-\frac{Q(y)}{D}}dy,\hspace{0.5cm}\hat{L}=\int
_{\mathbb{T}}e^{\frac{Q(y)}{D}}dy.
\]
and $\kappa=\frac{4\pi^{2}}{L\hat{L}}$. Now, by the results of \cite{DSW12} we obtain that the limiting LD related HJB equation boils down to
\begin{eqnarray}
\partial_{t}G(t,x)-\kappa x \partial_{x}G(t,x)- \kappa D|\partial_{x}G(t,x)|^{2}&=&0
\nonumber\\
G(T,x)&=&R(x).\nonumber
\end{eqnarray}

One can solve this equation explicitly and obtain
\begin{eqnarray}
G(t,x)&=&
\begin{cases}
\frac{(e^{\kappa T}-x e^{\kappa t})^2}{-2D e^{2\kappa t}+(1+2D)e^{2\kappa T}} & x\geq 0  \\
\frac{(e^{\kappa T}+x e^{\kappa t})^2}{-2D e^{2\kappa t}+(1+2D)e^{2\kappa T}} & x<0.
\end{cases}\nonumber
\end{eqnarray}

Notice that $G$ is not smooth at $x=0$. One can fit this problem into the subsolution framework by defining $G_{1}(t,x)=\frac{(e^{\kappa T}-x e^{\kappa t})^2}{-2D e^{2\kappa t}+(1+2D)e^{2\kappa T}}$ and
$G_{2}(t,x)=\frac{(e^{\kappa T}+x e^{\kappa t})^2}{-2D e^{2\kappa t}+(1+2D)e^{2\kappa T}}$ and then considering the subsolution $\check{G}(t,x)=\min\{G_{1}(t,x), G_{2}(t,x)\}$. In
general one should mollify it in order to produce a smooth subsolution (see \cite{DupuisSpiliopoulosZhou2013}), but it
is known (see \cite{VandenEijndenWeare} for an analogous situation) that mollification is not needed here since the  discontinuity is along only one interface.

By the results of \cite{DSW12}, the nearly optimal control is given by ${u}(t,x,y)=-\sqrt{2D}(1+\frac{\partial\chi}{\partial y}(y))\partial_{x}G(t,x)$.
Then observing that
\begin{equation*}
1+\frac{\partial \chi}{\partial y}(y)=\frac{2\pi}{\hat{L}}e^{Q(y)/D}=\frac{2\pi}{\hat{L}}e^{(\cos(y)+\sin(y))/D}
\end{equation*}
we obtain the expression for the optimal control
\begin{equation}
u_{LD}(t,x)\doteq{u}(t,x)=
\begin{cases}
-\sqrt{2D}\frac{2\pi}{\hat{L}} e^{(\cos(\frac{x}{\delta})+\sin(\frac{x}{\delta}))/D}\frac{-2 e^{\kappa t}( e^{\kappa T}-x e^{\kappa t})}{-2D e^{2\kappa t}+(1+2D) e^{2\kappa T}} & x> 0  \\
-\sqrt{2D}\frac{2\pi}{\hat{L}} e^{(\cos(\frac{x}{\delta})+\sin(\frac{x}{\delta}))/D}\frac{2 e^{\kappa t}( e^{\kappa T}+x e^{\kappa t})}{-2D e^{2\kappa t}+(1+2D) e^{2\kappa T}} & x<0.
\end{cases}\nonumber
\end{equation}

\subsubsection{Moderate deviations based scheme}\label{MDbasedISperiodic}
We first need to re-express the event of interest in terms of the moderate deviations scaling. For this purpose, we have
\begin{align}
\mathbb{E}\left[e^{-\frac{1}{\eps}R(X^{\eps}_{T})}\right]=\mathbb{E}\left[e^{-h^{2}(\eps)H(\eta^{\eps}_{T};\sqrt{\eps}h(\eps))}\right]\label{Eq:MD_Estimation}
\end{align}
where for $\beta>0$, we have defined
\begin{eqnarray}
H(\eta;\beta)&=&
\begin{cases}
\left(\eta-\frac{1-\bar{X}_{T}}{\beta}\right)^2 & \eta+\frac{\bar{X}_{T}}{\beta}\geq 0  \\
\left(\eta+\frac{1+\bar{X}_{T}}{\beta}\right)^2 & \eta+\frac{\bar{X}_{T}}{\beta}<0
\end{cases}\nonumber
\end{eqnarray}
and
\[
\eta^{\eps}_{t}=\frac{X^{\eps}_{t}-\bar{X}_{t}}{\sqrt{\eps}h(\eps)}.
\]

%

As in the previous section, we note the ambiguity of the dependence on the terminal condition in \eqref{Eq:MD_Estimation} on  $\eps$, i.e., $H(\eta;\sqrt{\eps}h(\eps))$. Now, in this problem if we choose $h(\eps)=1/\sqrt{\eps}$ then the problem becomes immediately an importance sampling problem with the large deviations scaling. Since, we are interested in seeing the effect of the moderate deviations scaling, we will choose different values of $h(\eps)$ and as we will present below the effect of the choice is minimal in the asymptotic regime.

The limiting MD related HJB equation  boils down to
\begin{eqnarray}
\partial_{t}G(t,\eta)-\kappa V''(\bar{X}(t))\eta \partial_{\eta}G(t,\eta)- \kappa D|\partial_{\eta}G(t,\eta)|^{2}&=&0
\nonumber\\
G(T,x)&=&H(\eta;\beta).\label{Eq:HJBequationExampleMD}
\end{eqnarray}

Ignoring the terminal condition from \eqref{Eq:HJBequationExampleMD}, we notice that  the main equation is the same for both large deviations and moderate deviations (notice that $V''(x)=1$). Next we assume that the initial point is exactly the stable equilibrium point, i.e.,  $x_{0}=0$. In this case, we actually have that
\begin{eqnarray}
H(\eta;\beta)&=&
\begin{cases}
\left(\eta-\frac{1}{\beta}\right)^2 & \eta\geq 0  \\
\left(\eta+\frac{1}{\beta}\right)^2 & \eta<0
\end{cases}\nonumber
\end{eqnarray}

and we obtain that a viscosity solution to \eqref{Eq:HJBequationExampleMD} is given by
\begin{eqnarray}
G(t,\eta)&=&
\begin{cases}
\frac{(\beta e^{\kappa T}-\eta e^{\kappa t})^2}{-2D e^{2\kappa t}+(1+2D)e^{2\kappa T}} & \eta\geq 0  \\
\frac{(\beta e^{\kappa T}+\eta e^{\kappa t})^2}{-2D e^{2\kappa t}+(1+2D)e^{2\kappa T}} & \eta<0.
\end{cases}\nonumber
\end{eqnarray}


By Theorem \ref{T:UniformlyLogEfficient}, the nearly optimal control is given by ${u}(t,\eta,y;\beta)=-\sqrt{2D}(1+\frac{\partial\chi}{\partial y}(y))\partial_{\eta}G(t,\eta;\beta)$ and we obtain the expression for the optimal control
\begin{equation}
u_{MD}(t,\eta, x;\beta)\doteq{u}(t,\eta, x/\delta;\beta)=
\begin{cases}
-\sqrt{2D}\frac{2\pi}{\hat{L}} e^{(\cos(\frac{x}{\delta})+\sin(\frac{x}{\delta}))/D}\frac{-2 e^{\kappa t}( \beta e^{\kappa T}-\eta e^{\kappa t})}{-2D e^{2\kappa t}+(1+2D) e^{2\kappa T}} & \eta> 0  \\
-\sqrt{2D}\frac{2\pi}{\hat{L}} e^{(\cos(\frac{x}{\delta})+\sin(\frac{x}{\delta}))/D}\frac{2 e^{\kappa t}( \beta e^{\kappa T}+\eta e^{\kappa t})}{-2D e^{2\kappa t}+(1+2D) e^{2\kappa T}} & \eta<0.
\end{cases}
\nonumber
\end{equation}

\subsubsection{Simulation Results for diffusion in rough potential}
In this section we summarize the results for the rough potential problem described in Subsection \ref{SS:PeriodicProblem}. Table \ref{Table1} has the related simulation results.

Let us choose $D=1$, initial point  $(t_{0},x_{0})=(0,0)$ and final time $T=1$. We  calculate that
\begin{equation*}
\hat{L}=9.84\textrm{ and } \kappa=\frac{4\pi^{2}}{L\hat{L}}=0.408.
\end{equation*}

Below we present results for the embedding (in the moderate deviations case)
\[
h(\eps)=\eps^{-0.4}
\]
which then implies that if $\delta$ decays sufficiently fast that
\[
\lim_{\eps\downarrow 0}j_{1}(\eps)=\lim_{\eps\downarrow 0}\frac{\delta/\eps}{\sqrt{\eps}h(\eps)}=0.
\]

For the sake of completeness, we mention here that we repeated the same simulation study but with $h(\eps)=\eps^{-0.1}$ and the results were statistically the same.

We used $N=5 \times 10^{6}$ trajectories with discretization step
\[
\text{T}_{\text{step}}=0.001\frac{\delta^{2}}{\eps}.
\]

\begin{table}[!ht]
\begin{center}
\begin{tabular}{|c|c|c|c|c|c|c|c|c|c|c|}
\hline
      $\eps$ & $\delta$& $\eps/\delta$ & $j_{1}=\frac{\delta/\eps}{\sqrt{\eps}h(\eps)}$ & $\hat{\theta}_{NMC}(\eps)$ & $\hat{\theta}_{LD}(\eps)$& $\hat{\theta}_{MD}(\eps)$ &$\hat{\rho}^{\eps}_{NMC}$   & $\hat{\rho}^{\eps}_{LD}$ & $\hat{\rho}^{\eps}_{MD}$     \\
     \hline  $0.25$ & $0.1$ &  $2.5$& $0.46$& $2.17e-01 $ &$2.18e-01$ & $2.17e-01 $ &$ 1.11$ & $7.86$ & $7.32$ \\
    \hline  $0.125$ & $0.04$  & $3.125$ &$0.39$&  $3.42e-02$ &$3.42e-02$ & $3.42e-02$ &$2.72$ & $4.28$ & $5.68$ \\
   \hline  $0.0625$ & $0.015$  & $4.167$ & $0.32$& $7.91e-04$ &$7.96e-04$ & $7.98e-04$ &$10.75$ & $3.50 $ & $3.43$ \\
    \hline  $0.03125$ & $0.007$  & $4.47$ &$0.31$ &$4.06e-07 $ &$4.54e-07$& $4.56e-07 $ & $104.23$ & $4.05$ & $ 4.56$ \\
    \hline  $0.025$ & $0.005$  & $5.0$ & $0.28$ &$9.51e-09 $ &$1.04e-08$ & $1.04e-08 $ &$248.73$ & $ 3.92$ & $ 5.36$ \\
    \hline  $0.02$ & $0.003$  & $6.06$ &$0.24$& $5.85e-11 $ & $9.03e-11$ & $9.0e-11 $ &$517.54$& $3.23$ & $5.53$ \\
 \hline
\end{tabular}
\end{center}
\caption{Comparison table for periodic diffusion in rough potential.\label{Table1}}
\end{table}

The conclusion from Table \ref{Table1} is that large deviations--based importance sampling works better if it can be done, but if the event is not too rare then moderate deviations--based importance sampling will also work well. The relative error for the moderate deviations estimator is larger than that of the large deviations estimator, and grows more rapidly as $\eps$ decreases. In comparison, both the moderate deviations estimator and the large deviations estimator have superior performance to the naive Monte Carlo estimator, for which the relative error grows rapidly and the estimator is no longer accurate for small $\eps$.

However, when dealing with multiscale systems it is rarely the case that one can actually write down subsolutions to the large deviations related HJB equations, but sometimes one can do so for moderate deviations--based importance sampling. We saw an example in this direction in Section \ref{SS:TwoScaleSystem_b}.  As discussed there, finding a subsolution for the large deviations related HJB equation would require resorting to numerical methods. However, it was straightforward to find an explicit solution to the moderate deviations related HJB equation.

\appendix
\section{Some useful lemmas}\label{S:AppendixA}

The following theorem collects results from~\cite{PV01} and~\cite{PV03} that are used in this paper.
\begin{theorem} [Results from~\cite{PV01} and~\cite{PV03}]\label{T:regularity}
Let Conditions \ref{C:growth} and \ref{C:ergodic} be satisfied. In Regime $i = 1, 2$ we have that,
\begin{enumerate}[(i)]
\item There exists a unique invariant measure $\mu_{i, x}(dy)$ associated with the operator $\opL_{i, x}$. For all $x \in \mathbb{R}^n$ and $q \in \mathbb{N}$,
\begin{equation*}
    \int_\mathcal{Y} \lvert y \rvert^q \,\mu_{i, x}(dy) < \infty.
\end{equation*}
Moreover, $\mu_{i, x}$ has a density which is twice differentiable in $x$.

\item Assume that $G(x, y) \in \mathcal{C}^{2, \alpha}(\mathbb{R}^n\times \mathcal{Y})$. Then
\begin{equation*}
    \bar{G}(x) = \int_\mathcal{Y} G(x, y) \,\mu_{i, x}(dy)
\end{equation*}
is twice differentiable in $x$.

\item Assume that $F(x, y) \in \mathcal{C}^{2, \alpha}(\mathbb{R}^n\times \mathcal{Y})$,
\begin{equation*}
    \int_\mathcal{Y} F(x, y) \,\mu_{i, x}(dy) = 0,
\end{equation*}
and that for some positive constants $K$ and $q_{F}$,
\begin{equation*}
   \lvert F(x, y) \rvert + \lVert \nabla_x F(x, y) \rVert + \lVert \nabla_x \nabla_x F(x, y) \rVert \le K (1 + \lvert y \rvert^{q_{F}}).
\end{equation*}
Then there is a unique solution from the class of functions which grow at most polynomially in $\lvert y \rvert$  to
\begin{equation*}
    \opL_{i, x} u(x, y) = - F(x, y), \quad \int_\mathcal{Y} u(x, y) \,\mu_{i, x}(dy) = 0.
\end{equation*}
Moreover, the solution satisfies $u(\cdot, y) \in \mathcal{C}^2$ for every $y \in \mathcal{Y}$, $\nabla_x \nabla_x u \in \mathcal{C}(\mathbb{R}^n\times\mathcal{Y})$, and there exist positive constant $K'$  such that
\begin{align*}
\lvert u(x, y) \rvert +  \lVert \nabla_y u(x, y) \rVert+ \lVert \nabla_x u(x, y) \rVert+ \lVert \nabla_x \nabla_x u(x, y) \rVert +  \lVert \nabla_x \nabla_y u(x, y) \rVert &\le K' (1 + \lvert y \rvert )^{q_{F}}.
\end{align*}
\end{enumerate}
\end{theorem}

Next, we recall some results from~\cite{MorseSpiliopoulos2017} related to certain bounds involving the controlled processes~\eqref{E:generalControlledSDE}. Notice that the lemmas below are proven in \cite{MorseSpiliopoulos2017} for the case of $u(s)=(0,0)$, but including $u(s)$ from Theorem \ref{T:UniformlyLogEfficient} does not change the proof. Also, we remark here that Lemma \ref{L:productBound} is based on Lemma \ref{L:uBound} and on the first statement of Lemma \ref{L:YIntegralgrowth} (see \cite{MorseSpiliopoulos2017} for the related details).

Lemma \ref{L:uBound} is standard, see for example Lemma B.1 of~\cite{MorseSpiliopoulos2017}.
\begin{lemma} \label{L:uBound}
Let Assumptions~\ref{C:growth} and \ref{C:ergodic} hold. Then the infimum of the representation in~\eqref{Eq:ToBeBounded} can be taken over all controls such that
\begin{equation*}
    \int_0^T  \lvert v^\eps(s)  \rvert^2 d s < R, \text{ almost surely},
\end{equation*}
where the constant $R<\infty$ does not depend on $\eps$ or $\delta$.
\end{lemma}

\begin{lemma}[Lemma B.3 in~\cite{MorseSpiliopoulos2017}] \label{L:productBound}
Let Assumptions~\ref{C:growth} and \ref{C:ergodic} hold.
Consider any family $\{v^{\eps},\eps>0\}$ of
controls in $\mathcal{A}$ satisfying, for some $R<\infty$,\begin{equation*}
    \sup_{\eps > 0} \int_0^T  \lvert v^\eps(s) \rvert^2  d s < R
\end{equation*}
almost surely.
Let $A(x, y)$ and $B(x,y)$ be given functions and $K$, $\theta\in(0,1)$ such that
\begin{equation*}
    |A(x, y)| \le K (1 + \lvert y \rvert^\theta), \text{ and }|B(x, y)| \le K (1 + \lvert y \rvert^{2\theta}).
\end{equation*}
Then for $\alpha \in \{1, 2\}$,
\begin{enumerate}[(i)]
\item for any $p\in(1,1/\theta]$, there exists $C<\infty$ such that
$$
\mathbb{E}\left(\sup_{t\in[0,T]}\left\lvert \int_{0}^{T} A(\hat{X}_s^{\eps, v^\eps}, \hat{Y}_s^{\eps, v^\eps}) v_\alpha^\eps(s) d s \right\rvert^{2p}
+ \sup_{t\in[0,T]}\left\lvert \int_{0}^{T} B(\hat{X}_s^{\eps, v^\eps}, \hat{Y}_s^{\eps, v^\eps})  d s \right\rvert^{p}\right)\le C;
$$
\item for any $p\in(1,1/\theta]$, there exists $C<\infty$ such that for fixed $\rho > 0$ and for all $0 \le t_1 < t_1+\rho \le T$,
$$
\mathbb{E}\left(\sup_{\substack{0 \le t_1 < t_2 \le T \\ \lvert t_2 - t_1 \rvert < \rho}}\left\lvert \int_{t_1}^{t_{2}} A(\hat{X}_s^{\eps, v^\eps}, \hat{Y}_s^{\eps, v^\eps}) v_\alpha^\eps(s) d s \right\rvert^{2p}
+
\sup_{\substack{0 \le t_1 < t_2 \le T \\ \lvert t_2 - t_1 \rvert < \rho}}\left\lvert \int_{t_1}^{t_{2}} B(\hat{X}_s^{\eps, v^\eps}, \hat{Y}_s^{\eps, v^\eps})  d s \right\rvert^{p}\right)
\le C \lvert \rho \rvert^{r/\theta-1};
$$
\item for all $\zeta > 0$,
\begin{equation*}
    \lim_{\rho\downarrow 0} \limsup_{\eps \downarrow 0} \mathbb{P} \left[ \sup_{\substack{0 \le t_1 < t_2 \le T \\ \lvert t_2 - t_1 \rvert < \rho}} \left\lvert \int_{t_1}^{t_2} A(\hat{X}_s^{\eps, v^\eps}, \hat{Y}_s^{\eps, v^\eps}) v_\alpha^\eps(s) d s \right\rvert > \zeta \right] = 0
\end{equation*}
and
\begin{equation*}
    \lim_{\rho \downarrow 0} \limsup_{\eps \downarrow 0} \mathbb{P} \left[ \sup_{\substack{0 \le t_1 < t_2 \le T \\ \lvert t_2 - t_1 \rvert < \rho}} \left\lvert \int_{t_1}^{t_2} B(\hat{X}_s^{\eps, v^\eps}, \hat{Y}_s^{\eps, v^\eps})  d s \right\rvert > \zeta \right] = 0.
\end{equation*}
\end{enumerate}
\end{lemma}

\end{document}